\documentclass[11pt]{article}

\usepackage{amsmath}
\usepackage{amsfonts}
\usepackage{graphicx, subfigure}
\usepackage[pdftex]{hyperref}
\usepackage[dutch,english]{babel}
\usepackage{amsmath,amssymb}
\usepackage{amsthm}
\usepackage[usenames,dvipsnames]{color}
\usepackage{tikz}
\usepackage{accents}
\usepackage{xcolor}
\usepackage{hyperref}
\usepackage{mathtools}
\usepackage{float}
\usepackage{multicol}
\usepackage{comment}
\usepackage[T1]{fontenc}
\usepackage{ stmaryrd }
\usepackage[font={small,it}]{caption}
\usetikzlibrary{decorations.pathmorphing}
\usepackage{fullpage}\usepackage{graphicx}

\DeclareMathOperator{\id}{Id}

\DeclareMathOperator{\im}{Im}

\DeclareMathOperator{\Id}{Id}

\definecolor{bloesh}{RGB}{160, 194, 185}
\definecolor{roo}{RGB}{241,23,70}

\newfont{\gothic}{eufm10}
 
\def\x{{\bf x}}

\def\z{{\bf z}}

\theoremstyle{plain}
\newtheorem{thr}{Theorem}[section]

\newtheorem{lem}[thr]{Lemma}
\newtheorem{cor}[thr]{Corollary}
\newtheorem{prop}[thr]{Proposition}
\theoremstyle{definition}
\newtheorem{defi}[thr]{Definition}
\newtheorem{definition}[thr]{Definition}
\newtheorem{ex}[thr]{Example}
\theoremstyle{remark}
\newtheorem{remk}[thr]{Remark}
\newtheorem{rem}[thr]{Remark}
\newtheorem{remark}[thr]{Remark}
\theoremstyle{remark}

\newcommand{\field}[1]{\mathbb{#1}}
\newcommand{\R}{\field{R}}

\newcommand{\Z}{\field{Z}}

\definecolor{bloesh}{RGB}{160, 194, 185}
\definecolor{roo}{RGB}{241,23,70}

\usepackage[framemethod=TikZ]{mdframed}
\mdfdefinestyle{MyFrame}{%
    linecolor=orange,
    outerlinewidth=1pt,
    roundcorner=3pt,
    innertopmargin=.2\baselineskip,
    innerbottommargin=.2\baselineskip,
    innerrightmargin=20pt,
    innerleftmargin=20pt,
    backgroundcolor=blue!80!white}

\newcommand{\gup}[1]{G^\uparrow_{#1}}
\newcommand{\gdown}[1]{G^\downarrow_{#1}}

\newcommand{\0}{{\bf 0}}

\newcommand{\lup}[1]{L^\uparrow_{#1}}
\newcommand{\ldown}[1]{L^\downarrow_{#1}}

\newcommand{\bm}[1]{B_{#1}}
\newcommand{\bt}[1]{B_{#1}^\intercal}
\newcommand{\fup}[1]{F^{\uparrow}_{#1}}
\newcommand{\fdown}[1]{F^{\downarrow}_{#1}}
\newcommand{\cw}[2]{\mathfrak{C}_{#1}^{#2}}

\newcommand{\ccup}[1]{K^{\uparrow}_{#1}}
\newcommand{\ccdown}[1]{K^{\downarrow}_{#1}}
\newcommand{\ccupv}[1]{\bar{K}^{\uparrow}_{#1}}
\newcommand{\ccdownv}[1]{\bar{K}^{\downarrow}_{#1}}

\newcommand{\mc}[1]{\mathcal{#1}}
\renewcommand{\phi}{\varphi}

\title{Dynamical systems defined on simplicial complexes:  symmetries, conjugacies, and invariant subspaces}
\date{\today}
\author{Eddie Nijholt\footnote{ICMC S\~{a}o Carlos, Universidade de S\~{a}o Paulo, Av. Trab. S\~{a}o Carlense 400, S\~{a}o Carlos - SP, 13566-590, Brasil} and Lee DeVille\footnote{Department of Mathematics, University of Illinois, 1409 W. Green St., Urbana IL 61801 USA}}

\DeclareMathOperator{\sgn}{sgn}

\begin{document}

\maketitle

\begin{abstract}
  We consider the general model for dynamical systems defined on a simplicial complex.  We  describe the conjugacy classes of these systems and show how symmetries in a given simplicial complex manifest in the dynamics defined thereon, especially with regard to invariant subspaces in the dynamics.

\end{abstract}

Lead Paragraph:  In this paper we study the general form of a nonlinear dynamical system defined on a simplicial complex, which is a natural higher-order generalization of dynamics on a network.  We first show that many of the important algebraic homological structures survive when we consider nonlinear systems.  We give a full description of the conjugacy class of dynamical systems that can be defined on a simplicial complex, and also give a process by which one can construct a system on a given complex that is equivalent to any given dynamical system.  We show that the choice of orientation matters for simplicial complexes and understand the impact of changing orientations, and finally we study the (dynamical) symmetries of dynamical systems defined on simplicial complexes that have their own (algebraic) symmetries.

\section{Introduction}

\subsection{Background}

The field of network science has grown significantly over the last few decades~\cite{newman2006structure, boccaletti2006complex, barabasi2016network, newman2018networks}, but more recently, in just the last few years, there has been significant activity in both the modeling and empirical communities that  we should not just consider pairwise interactions, but  we should also pay attention to interactions of larger subsets of elements~\cite{salnikov2018simplicial, mulder2018network, lambiotte2019networks, torres2020and, serrano2020simplicial, giusti2016two, bassett2017network, curto2017can, sizemore2018cliques, kanari2018topological, klamt2009hypergraphs, ladroue2009beyond, roman2015simplicial, courtney2016generalized, patania2017shape, iacopini2019simplicial, muhammad2006control, stolz2017persistent, benson2018simplicial, barbarossa2020topological, aguiar2022network}. For  very nice and comprehensive surveys of the state of the art, see~\cite{battiston2020networks, porter2020nonlinearity+, bianconi2021higher}.  Another way of saying this is that we should not just consider dynamics on graphs but also on hypergraphs~\cite{berge1984hypergraphs, bollobas1986combinatorics}.  A special subclass of hypergraphs is the {\em simplicial complexes}; these are hypergraphs with the property that if a subset is active, then all of its subsets are also active.  These are the structures we discuss in this paper.  The study of simplicial complexes has a long history in mathematics~\cite{Munkres.book, Hatcher.book, lim2020hodge} and underpins many modern tools in algebraic topology and homological algebra.

In some form or another, the main question from the field of network dynamical systems asks how a network structure influences dynamical behavior. Some of the most successful approaches consist of identifying some combinatorial or algebraic structure in a given network, and translating this structure into necessary geometric features of the dynamics. Notable in this respect is the groupoid formalism of Golubitsky, Stewart et al \cite{stewart2003symmetry, golubitsky2005patterns, golubitsky2006nonlinear}, which allows one to predict robust synchrony patterns, among other dynamical properties. Another clear example is the translation of a symmetry in the network graph into a symmetry of the corresponding dynamical system \cite{golubitsky1999symmetry, dias2006local}. Other methods use graph-fibrations \cite{deville2015modular, nijholt2016graph} and generalized or hidden symmetries \cite{antoneli2008hopf, nijholt2019center, nijholt2020quiver}.

The main questions that we would like to answer in this paper are:

\begin{enumerate}

\item What is the  most general framework for (ODE) dynamics that can be said to be consistent with a simplicial complex structure?

\item When the dynamics live on a simplicial complex, and this complex has certain algebraic symmetries, what does this imply for the \{symmetries/quasi-symmetries/etc.\} of the dynamics?

\end{enumerate}

Finally, the main point:  we seek to do for simplicial dynamics what previous authors have done for network dynamics.

Towards this goal, we will first show that the linear structure that is commonly used in homological and cohomological computations both survives in, and is important in understanding, the ODE models on simplicial complexes (Section~\ref{sec:triple}). We then investigate the conjugacy classes of dynamical systems on a specific simplicial complex (Sections~\ref{sec:realization},~\ref{sec:conjugacies}), the impact of choices of orientation on these dynamical systems (Section~\ref{sec:orientation}), and finally finish with a study of the symmetries and invariant spaces of dynamical systems on simplicial complexes (Sections~\ref{sec:symmetries}~\ref{sec:invariant}).

\subsection{General model definition}\label{sec:def}

As described above, we present the general definition of the flow we study.

\newcommand{\dmax}{d_{{\mathsf {max}}}}
\newcommand{\uppadj}{\smallfrown}

\newcommand{\lowadj}{\smallsmile}

\begin{defi}[Simplicial complex]
Let $V$ be a set.  A {\bf $k$-simplex} is a unordered set $\{v_0,\dots,v_k\}$ with $v_i\in V$ and $v_i\neq v_j$ if $i \neq j$.  A {\bf face} of a $k$-simplex is all $(k-1)$-simplices of the form $\{v_0,\dots,v_k\}\setminus \{v_i\}$, which we will also denote by $\{v_0,\dots,\widehat{v_i},\dots,v_k\}$.  A {\bf simplicial complex} $X$ is a collection of simplices closed under inclusion of all faces, and we denote by $X_d$ the set of simplices that contain precisely $d+1$ elements. We will also say that the simplices in $X_d$ are of dimension $d$.  We denote by $\dmax(X)$ the maximal dimension of all of the simplices in $X$.

Two $d$-simplices $F_1,F_2$ of a complex $X$ are called {\bf upper adjacent} if both are faces of some $d+1$-simplex; in this case we write $F_1 \uppadj F_2$. 
Two $d$-simplices $F_1,F_2$ of a complex $X$ are called {\bf lower adjacent} if they have a common face; in this case we write $F_1 \lowadj F_2$.

An ordering of vertices gives an {\bf orientation}, and an {\bf oriented $k$-simplex} is a $k$-simplex with an order type.  By definition, orientation is anti-symmetric with respect to an exchange of vertices.  We denote an ordered simplex by square brackets, e.g. $[v_0,\dots, v_k]$.  \hfill $\triangle$

\end{defi}

For $X$ a simplicial complex, let $C_d(X)$ be the $\R$-vector space with basis given by the elements of $X_d$.  For each $d$, we have the {\bf boundary map} defined by 
\begin{align*}
  \partial_d \colon C_d(X) &\to C_{d-1}(X)\\
    [v_0,\dots,v_d]&\mapsto \sum_{\ell=0}^d (-1)^\ell [v_0,\dots,\widehat{v_\ell},\dots, v_d]
\end{align*}
and extended by linearity.  (As is well known, the map $\partial_d\circ \partial_{d+1}$ is identically zero.)  The elements of the kernel of the map $\partial_d\colon C_d(X)\to C_{d-1}(X)$ are known as the $d$-cycles of the complex $X$, and the image of $\partial_{d+1}\colon C_{d+1}(X)\to C_d(X)$ consists of the $d$-boundaries.  Since $\partial_d\circ \partial_{d+1}=0$, it is natural to define the $d$-th homology group $ H_d(X) := \ker \partial_d / \im \partial_{d+1}$.  There is also a natural inner product on $C_d(X)$; with respect to this inner product we can define the adjoint map $\partial^*_{d}\colon C_{d-1}(X)\to C_d(X)$.  We also define the $d$-cocyles as $\ker(\partial^*_{d+1})$ and $d$-coboundaries as $\im(\partial^*_{d})$. \\

Choosing an ordering of the elements of $X_d$ naturally chooses an ordered basis for $C_d(X)$.  With this ordered basis we can define the matrix $B_d$ to be the matrix representation of $\partial_d$.  We assume implicitly in all that follows that such an ordering has been chosen and fixed for $X_d$.

The most general ODE we will consider here in this paper is given by
\begin{align}\label{eq:1}
\frac{d}{dt}\theta_d = G^0_d(\theta_d) + B_{d}^\intercal\fdown d(B_{d}\theta_d) + B_{d+1}\fup d(B_{d+1}^\intercal\theta_d) \, 
\end{align}
where 
\begin{equation*}
  \fdown d\colon C_{d-1}(X)\to C_{d-1}(X),\quad \fup d\colon C_{d+1}(X)\to C_{d+1}(X)
\end{equation*}
can be arbitrary nonlinear functions. The function $G^0_d$ satisfies the condition that the $i$th coordinate of $G^0_d$ depends only on the $i$th coordinate of its input.  We will write
\begin{equation*}
  \gdown d := B_{d}^\intercal\fdown d B_{d}, \quad\gup d = B_{d+1}\fup d B_{d+1}^\intercal, \quad  G_d:=\gdown d + \gup d,
\end{equation*}
thus giving the more compact description of~\eqref{eq:1}:
\begin{equation}
  \frac{d}{dt}\theta_d = G^0_d(\theta_d)  + \gdown d(\theta_d) + \gup d(\theta_d).
\end{equation}

The main motivation behind this form of equation is that each of the terms represents a different type of interaction.  The first term, $G_d^0(\cdot)$, represents ``internal dynamics'': we allow a simplex to depend on its own value.  The next term, $\gup d(\cdot)$, represents an ``up-coupling'':  each simplex depends on the others only through couplings via a higher-dimensional simplex, and similarly for the ``down-coupling''.  

Finally, although we allow for general nonlinear functions $\fdown d$ and $\fup d$ in~\eqref{eq:1}, we often consider one class of functions in particular, defined as:

\begin{defi}
  We call $f\colon \R^n\to \R^n$ a {\bf componentwise} function if the $k$th component function of $f$ is a function only of the $k$th variable, i.e. if 
  \begin{equation*}
    f_k(x_1,\dots, x_n) = f_k(x_k).
  \end{equation*}
Note that $G^0_d$ is assumed componentwise above, and we will see that we can obtain some strong results when $\fup d$ and $\fdown d$ are assumed to be componentwise. \hfill $\triangle$
  \end{defi}

\section{The triple decomposition}\label{sec:triple}

In this section we present a natural ``triple decomposition'' of the space $C_d(X)$.  This decomposition is standard, but here we show that it naturally interacts well with the nonlinear flow.

\begin{defi}
We denote $W_d := \ker(B_d) \cap  \ker(B_{d+1}^\intercal) \subseteq C_d(X)$.  \hfill $\triangle$
\end{defi}

\begin{lem}\label{spilttingseverywhere}\label{lem:decomp}
For all $d$ we have the orthogonal decomposition
\begin{equation}\label{eq:triple}
  C_d(X)=  \im(B_d^\intercal) \oplus \im(B_{d+1}) \oplus  W_d,
\end{equation}
which we refer to as the {\bf triple decomposition} below.  We also have 
\begin{equation*}
  \ker(B_d) =  \im(B_{d+1}) \oplus W_d, \quad \ker(B_{d+1}^\intercal)  = \im(B_d^\intercal) \oplus W_d
\end{equation*}
from which we obtain the useful orthogonal decompositions
\begin{align}
  C_d(X) &=  \im(B_d^\intercal) \oplus \ker(B_d)\label{eq:downdecomp},\\
  C_d(X) &=  \im(B_{d+1}) \oplus \ker(B_{d+1}^\intercal)\label{eq:updecomp}.
\end{align}
\end{lem}

\begin{remark}
We leave the proof of~\ref{lem:decomp} to Appendix~\ref{app:proofs}, but note that all of these statements follow from two observations:  the Rank--Nullity theorem, plus the fact that $B_dB_{d+1} = 0$.\end{remark}

\begin{defi} 
We denote by $P_d$, $Q_d$ and $R_d$ the orthogonal projections onto $\im(B_d^\intercal)$, $\im(B_{d+1})$ and $W_d$, respectively in the triple decomposition~\eqref{eq:triple}.  It follows directly from Lemma~\ref{lem:decomp} that $P_d+Q_d+R_d= I_d := \Id_{C_d(X)}$. We also see that the pair of projections $(P_d, I_d-P_d)$ gives the decomposition~\eqref{eq:downdecomp}, and the pair $(Q_d, I_d-Q_d)$ gives the decomposition~\eqref{eq:updecomp}.  \hfill $\triangle$
\end{defi}

It follows directly from the definitions above that $P_d \bt d = \bt d$ and $Q_{d}B_{d+1} = B_{d+1}$.  We also note that since the projections are orthogonal, they are self-adjoint as operators, and from this $B_dP_d = B_d$ and $\bt{d+1}Q_{d} = \bt{d+1}$.  This means that 
\begin{equation}\label{eq:3=5}
  \gup d = B_{d+1}\fup d B_{d+1}^\intercal = B_{d+1}(P_{d+1}\fup d P_{d+1})B_{d+1}^\intercal.
\end{equation}
(Of course, the analogous argument applies, {\em mutatis mutandis}, for $\fdown d$ and $Q_{d-1}$.) We formalize this:

\newcommand{\equp}[1]{\sim^\uparrow_{#1}}
\newcommand{\eqdown}[1]{\sim^\downarrow_{#1}}

\begin{defi}
Consider the space of all smooth maps from $C_{d}(X)$ to itself.  We define two equivalence relations on this set:  $\equp d, \eqdown d$, defined by
\begin{align*}
  \mc A \equp d \mc B &\iff P_{d} \mc A P_{d} = P_{d} \mc B P_{d},\\
  \mc A \eqdown d \mc B &\iff Q_{d} \mc A Q_{d} = Q_{d} \mc B Q_{d}.
\end{align*}
 \hfill $\triangle$
\end{defi}

\begin{lem}\label{lem:1-to-1}
There is a one-to-one correspondence between the equivalence classes of $\equp {d+1}$ and distinct maps of the form $\gup d = B_{d+1}\fup d B_{d+1}^\intercal$, and the bijection is realized by the map $[\fup d] \mapsto B_{d+1}\fup d B_{d+1}^\intercal$.  

Similarly, there is a one-to-one correspondence between the equivalence classes of $\eqdown {d-1}$ and distinct maps of the form $\gdown d = B_{d}^\intercal\fdown d B_{d}^\intercal$, and the bijection is realized by the map $[\fdown d] \mapsto B_{d}^\intercal\fdown d B_{d}^\intercal$.  

\end{lem}

\begin{proof}
We will prove the statement for $\equp {d+1}$, and the second is similar.

It is clear from~\eqref{eq:3=5} that $\mc A\equp {d+1} \mc B$ implies $B_{d+1} \mc A B_{d+1}^\intercal = B_{d+1} \mc B  B_{d+1}^\intercal$. Now, assume that $\mc A, \mc B$ satisfy $B_{d+1} \mc A B_{d+1}^\intercal = B_{d+1} \mc B  B_{d+1}^\intercal$.  Define $\mc C = \mc A - \mc B$, and then  $B_{d+1} \mc C B_{d+1}^\intercal = 0$. Using~\eqref{eq:3=5} again, we see that $B_{d+1} P_{d+1} \mc C P_{d+1} B_{d+1}^\intercal = 0$ as well. Now, Lemma \ref{spilttingseverywhere} tells us that 
\begin{equation*}
  C_{d+1}(X) =   \im(B_{d+1}^\intercal) \oplus \ker(B_{d+1}) = \im(P_{d+1}) \oplus \ker(B_{d+1}),
\end{equation*}
so that $B_{d+1}$ is injective on $\im(P_{d+1})$. It follows that $P_{d+1} \mc C P_{d+1} B_{d+1}^\intercal = 0$, so that $P_{d+1} \mc C P_{d+1}$ vanishes on $\im(B_{d+1}^\intercal) =  \im(P_{d+1})$. As $P_{d+1}$ is a projection, we conclude that 
\begin{equation*}
 P_{d+1} \mc C P_{d+1} = P_{d+1} \mc C P_{d+1}^2 = (P_{d+1} \mc C P_{d+1}) |_{\im(P_{d+1})}P_{d+1} = 0.
\end{equation*}
This implies $P_{d+1} \mc A P_{d+1}=P_{d+1} \mc B P_{d+1}$ and thus $\mc A\equp {d+1} \mc B$.
\end{proof}

Consider the decomposition of $C_{d+1}(X)$ generated by the pair of projections $(P_{d+1},I_{d+1}-P_{d+1})$; writing $\theta_{d+1} = (x,y) := (P_{d+1}\theta_{d+1}, (I_{d+1}-P_{d+1})\theta_{d+1})$ and $\fup d = (g,h):=(P_{d+1}\fup d,(I_{d+1}-P_{d+1})\fup d)$, then generally we have
\begin{equation*}
  \fup d = \begin{pmatrix}
g(x,y)\\
h(x,y)
\end{pmatrix},\quad 
  P_{d+1}\fup d  P_{d+1} = \begin{pmatrix}
g(x,0)\\
0
\end{pmatrix}\, .
\end{equation*}
From~\eqref{eq:3=5} we see that these ``versions'' of $\fup d$ will give the same $\gup d$, and so far as we are concerned they are equivalent.  Clearly, the function with the projections is simpler, and we would like to use it as a canonical representative of all $\fup d$ that give the same vector field, giving the following definition:

\begin{defi}
  The {\bf canonical representative} of an equivalence class $[\mathcal{A}]$ for $\equp d$ is the function $P_d\mathcal{A}P_d$; \\ the {\bf canonical representative} of an equivalence class $[\mathcal{A}]$ for $\eqdown d$ is the function $Q_d\mathcal{A}Q_d$.  \hfill $\triangle$
\end{defi}

\begin{remk} Note that there is a one-to-one correspondence between equivalence classes of $\sim^{\uparrow}_{d+1}$ and maps from $\im(B_{d+1}^\intercal)$ to itself. More precisely, to a map $f: \im(B_{d+1}^\intercal) \rightarrow \im(B_{d+1}^\intercal)$ we may associate the class under $\sim^{\uparrow}_{d+1}$ of 
\begin{equation}\label{easierformqfq}
\fup d  = \begin{pmatrix}
f(x)\\
0
\end{pmatrix}: C_{d+1}(X) \rightarrow C_{d+1}(X)\, ,
\end{equation}
where the two components of the vector correspond to the decomposition 
\[ C_{d+1}(X) =   \im(B_{d+1}^{\intercal}) \oplus \ker(B_{d+1}) \, . \] 
This bijectively sends the maps $f: \im(B_{d+1}^\intercal) \rightarrow \im(B_{d+1}^\intercal)$ to the canonical representatives for $\equp {d+1}$.
Likewise, there is a correspondence between equivalence classes of  $\sim^{\downarrow}_{d-1}$ and maps from $\im(B_{d})$ to itself, by sending a map $g: \im(B_{d}) \rightarrow \im(B_{d})$ to the class under $\sim^{\downarrow}_{d-1}$ of 
\begin{equation}\label{easierformqfq2}
\fdown d  = \begin{pmatrix}
g(x)\\
0
\end{pmatrix}: C_{d-1}(X) \rightarrow C_{d-1}(X)\, ,
\end{equation}
where the two components correspond to the decomposition 
\[ C_{d-1}(X) =   \im(B_{d}) \oplus \ker(B_{d}^\intercal) \, . \] 
This bijectively sends the maps $g: \im(B_{d}) \rightarrow \im(B_{d})$ to the canonical representatives for $\sim^{\downarrow}_{d-1}$. In conclusion, even though $\gup d$ and $\gdown d$ are both vector fields on $C_d(X)$, we see that they are uniquely determined by vector fields on $\im(B_{d+1}^{\intercal})$ and $\im(B_d)$, respectively.  \hfill $\triangle$
\end{remk}
\section{Realization of specific vector fields}\label{sec:realization}

\newcommand{\pkl}[2]{P_{#1}^{#2}}

A natural question arises:  let us say that we have a particular vector field in mind, is it possible to realize it as a flow on a simplicial complex?  More specifically, we would like to describe the conjugacy class of all vector fields of the form
\begin{align}\label{eq:1'}
\frac{d}{dt}\theta_d = B_{d}^\intercal\fdown d(B_{d}\theta_d) + B_{d+1}\fup d(B_{d+1}^\intercal\theta_d) .
\end{align}
Also, we want to generate a process by which we can determine how to choose $\fup d, \fdown d$, to get the vector field we desire.

\begin{defi}[Moore--Penrose psuedoinverse]~\cite[Section 7.3]{Horn.Johnson.book}\label{moorepenrose}
  If $L\colon V\to W$ is a linear operator, then the linear operator $L^+\colon W \to V$ satisfying:
\begin{enumerate}

\item $LL^+L = L$;
\item $L^+ L L^+ = L^+$;
\item $LL^+$ and $L^+L$ are positive semi-definite.
\end{enumerate}
is known as the {\bf Moore--Penrose psuedoinverse} (or, more briefly, the {\bf psuedoinverse}) of $L$.  \hfill $\triangle$
\end{defi}

\begin{lem}\label{lem:oplus}
  For any $L\colon V\to W$, the pseudoinverse exists and is unique.  Moreover, we have the following orthogonal decompositions:
  \begin{equation*}
    V = \ker L \oplus \im L^+,\quad\quad W = \ker L^+ \oplus \im L.
  \end{equation*}
  Finally, $(L^*)^+ = (L^+)^*$, and we write $L^{+*}$ to be either operator.
\end{lem}

We prove this lemma below in Appendix~\ref{app:proofs}.

\begin{rem}
  By choosing a basis for $V$ and $W$, we can derive the analogous statements about matrices. \hfill $\triangle$
\end{rem}

\begin{defi}
  Given a vector space $V$ and a decomposition $V = V_1 \oplus V_2$, we say that $F\colon V\to V$ is {\bf independent of $V_1$} if, whenever $v = v_1+v_2$ with respect to the decomposition, we have $F(v) = F(v_2)$. \hfill $\triangle$
\end{defi}

\begin{lem}\label{lem:FofG}
  Let $L\colon V\to W$ be a linear map and assume that $G\colon W\to W$ has the property that $G$ is independent of $\ker L^*$ with respect to the decomposition $W = \ker L^* \oplus \im L^{+*}$, and $G(W)\subseteq \im L$.  Then there exists a function $F\colon V\to V$ such that $G(x) = LF(L^* x)$.    
  More specifically, we can choose 
\begin{equation}\label{eq:defofF}
  F(y) := L^+G(L^{+*} y)
\end{equation}  
as such a function (although this is typically not the only choice that works).
\end{lem}

\begin{proof}
  Since $G(W)\subseteq \im L$, we can write $G(x) = L\mc G(x)$ for some function $\mc G$.  Since $G$ is independent of $\ker L^*$, we can consider the decomposition $W = \ker L^* \oplus \im L^{+*}$.  If we let $x = x_1+x_2$ in this decomposition, then $G(x) = G(x_2)$, and as such we can assume that the argument of $G$ lies in $\im L^{+*}$.  In summary, we can write $G(x) = L\mc G(L^{+*} y)$ for some $y\in V$.
  
  Now assume the choice made in~\eqref{eq:defofF}, and we compute
  \begin{align*}
    LF(L^*x) = LL^+G(L^{+*}L^*x) = LL^+L\mc G(L^{+*}L^*L^{+*} y) = L\mc G(L^{+*}y) = G(x).
  \end{align*}

\end{proof}

\begin{prop}\label{prop:invariant-converse}
Let $G_1\colon \im(B_d^\intercal)\to \im(B_d^\intercal)$ and $G_2\colon \im(B_{d+1})\to \im(B_{d+1})$, then $G_1\oplus G_2\oplus \0_{W_d} $ is the right-hand side of~\eqref{eq:1'} for some choices of $\fup d, \fdown d$.
\end{prop}

\begin{proof}
  Let us write $U_d = \im(B_d^\intercal)$ and $V_d = \im(B_{d+1})$.    For each of $G_1, G_2$, we extend their domains.   For $x\in C_d(X)$, write $x = x_1 + x_2 \in U_d \oplus U_d^\perp$, and define $\widetilde{G}_1(x) := G_1(x_1)$.  Similarly, for $y\in C_d(X)$, write $y = y_1 + y_2 \in V_d \oplus V_d^\perp$, and define $\widetilde{G}_2(x) := G_2(x_2)$.
  
Note that $\widetilde{G}_1$ satisfies the assumptions of Lemma~\ref{lem:FofG} for the linear operator given by $\bm {d}^{\intercal}$; clearly $\widetilde{G}_1(C_d(X))\subseteq U_d$ by definition, and $U_d^\perp = (\im(B_d^\intercal))^\perp = \ker B_{d}$.   As such there is $\fdown d\colon C_{d-1}(X)\to C_{d-1}(X)$ giving $\widetilde{G}_1 = B_{d}^{\intercal}\fdown d \bm {d}$.  The argument is similar for $\widetilde{G}_2$, which gives the result.
\end{proof}

  \begin{definition}\label{def:abn}
 We call $H\colon \R^{n}\to\R^{n}$ an $(\alpha,\beta,n)$-vector field (for $\alpha+\beta \le n$)  if there exist $H_1\colon \R^\alpha\to\R^\alpha$ and $H_2\colon \R^\beta\to\R^\beta$ such that 
  \begin{equation*}
    H = H_1 \oplus H_2 \oplus H_3, 
  \end{equation*}
  where $H_3$ is the zero function on $\R^{n-\alpha-\beta}$. \hfill $\triangle$
\end{definition}

\begin{prop}\label{thm:conj}
   The conjugacy classes of all systems of the form~\eqref{eq:1'} are exactly the $(r_d, r_{d+1}, n_d)$-vector fields, where $r_{d} = \dim(\im(B_{d}^{\intercal}))$, $r_{d+1} = \dim(\im(B_{d+1}))$,  and $n_d = \dim(C_d(X))$.  More specifically, every vector field of the form~\eqref{eq:1'} is conjugate to an  $(r_d, r_{d+1}, n_d)$-vector field; conversely, any $(r_d, r_{d+1}, n_d)$-vector field is conjugate to~\eqref{eq:1'} for some choice of $\fdown d(\cdot), \fup d(\cdot)$.
\end{prop}

\begin{proof}
  Let $H(\cdot)$ be  a fixed $(r_d, r_{d+1}, n_d)$-vector field, and suppose
  \begin{equation*}
  M_1 \colon \im(\bt d) \to \R^{r_d},\quad   M_2\colon \im(B_{d+1})\to \R^{r_{d+1}},\quad  M_3\colon W_d \to \R^{n_d-r_d-r_{d+1}},
  \end{equation*}
  are given isomorphisms.  Define $G_i(x) = M^{-1}_iH_i(M_ix)$ for $i = 1,2,3$ and $G(x) = M^{-1}H(Mx)$, where $M = M_1 \oplus M_2 \oplus M_3$. By Lemma ~\ref{spilttingseverywhere} we have $G = G_1\oplus G_2 \oplus G_3$ (where the decomposition here is the triple decomposition of~\eqref{eq:triple}).   Note that $G_1\colon \im(B_{d}^{\intercal}) \to \im(B_{d}^{\intercal}) $, $G_2\colon \im(B_{d+1})\to \im(B_{d+1})$ and $G_3 \colon W_d \to W_d$.
  
  Clearly, $G$ and $H$ are conjugate, and we have $G_3 = \0_{W_d}$.  Using Proposition~\ref{prop:invariant-converse}, we see that $G(x)$ can be put into the form of~\eqref{eq:1'}.
  
Conversely, assume that $G(x)$ is in the form of~\eqref{eq:1'}.  Since $B_{d}$ vanishes on $ \ker(B_d) =  \im(B_{d+1}) \oplus W_d$ and $B^{\intercal}_{d+1}$ vanishes on $ \ker(B_{d+1}^\intercal)  = \im(B_d^\intercal) \oplus W_d$, we see that
$G = \gdown d|_{\im(B_d^\intercal)} \oplus \gup d|_{\im(B_{d+1})} \oplus \0_{W_d}$.  Writing $H_i(x) = M_iG_i(M_i^{-1}x)$ and $H = H_1\oplus H_2 \oplus H_3$ shows that $H$ is a $(r_d, r_{d+1}, n_d)$-vector field, and we are done.
\end{proof}

\begin{rem}
  Said colloquially, the conjugacy class of~\eqref{eq:1'} is the set of vector fields which decompose into two vector fields, one of dimension $r_d$ and the other of dimension $r_{d+1}$, in an $n_d$-dimensional vector space. \hfill $\triangle$
\end{rem}

\begin{ex}[Special case of one term]

Let us consider the case where there is a single ``up'' or ``down'' term on the right-hand side of Equation~\eqref{eq:1'} .  This happens, for example, whenever we consider the ``top-dimensional'' flow, or the vertex flow.     The top-dimensional flow on a simplicial complex of dimension $D$ is:
\begin{equation}\label{eq:triangle}
  \theta_D' = B_D^\intercal \fdown D(B_D \theta_D).
\end{equation}
Most of the results in this section carry over with the appropriate modifications. The main difference here is that we have the decomposition $C_D(X) = \im \bt D \oplus \ker \bm D$, and thus the conjugacy class is the set of $(r_D, 0, n_D)$-vector fields. \hfill $\triangle$

\end{ex}

\section{Conjugacies using the Laplacian}\label{sec:conjugacies}

\noindent The term $\gup d = B_{d+1}\fup d B_{d+1}^\intercal$ can be seen as a distortion of $\fup d$, and similarly for $\gdown d$ with respect to $\fdown d$. The following useful lemma paints a clearer picture of this distortion. 

\begin{lem}\label{conjugacieseverywhe}
The system \[\dot{\theta} = \gup d(\theta) = B_{d+1}\fup d B_{d+1}^\intercal(\theta)\, ,\] seen as an ODE on $\im(B_{d+1})$, is conjugate to both systems on $\im(B_{d+1}^\intercal)$:
\[\dot{x} = (B_{d+1}^{\intercal}B_{d+1})\fup d(x)  \quad \text{ and } \quad \dot{y} = P_{d+1}\fup d(B_{d+1}^{\intercal}B_{d+1}y) \, ,\]
where $x,y \in \im(B_{d+1}^\intercal)$. \\
Similarly, the system \[\dot{\theta} = \gdown d(\theta) = B_{d}^\intercal\fdown d B_{d}(\theta)\, ,\] seen as an ODE on $\im(B_{d}^\intercal)$, is conjugate to both systems on $\im(B_{d})$:
\[\dot{x} = (B_d B_{d}^{\intercal})\fdown d(x) \quad \text{ and } \quad \dot{y} = Q_{d-1}\fdown d(B_d B_{d}^{\intercal} y) \, , \]
where $x,y \in \im(B_{d})$.
\end{lem}

\begin{proof}
By Lemma \ref{spilttingseverywhere} we have the splitting
\[C_{d}(X) =  \im(B_{d+1}) \oplus \ker(B_{d+1}^\intercal) \, ,\]
so that the restriction of $B_{d+1}^\intercal$ to $\im(B_{d+1})$ is injective. Hence, we may view $B_{d+1}^\intercal |_{\im(B_{d+1})}$ as a bijection between $\im(B_{d+1})$ and $\im(B_{d+1}^\intercal)$. As such, we claim that it conjugates $B_{d+1}\fup d B_{d+1}^\intercal$ and $B_{d+1}^{\intercal}B_{d+1}\fup d$. A straightforward calculation indeed shows that
\begin{align}
(B_{d+1}^\intercal)(B_{d+1}\fup d B_{d+1}^\intercal) = B_{d+1}^\intercal B_{d+1}\fup d B_{d+1}^\intercal = (B_{d+1}^{\intercal}B_{d+1}\fup d)(B_{d+1}^\intercal)\, .
\end{align}
Likewise, the splitting 
\[C_{d+1}(X) = \im(B_{d+1}^\intercal) \oplus \ker(B_{d+1})  \]
tells us that $B_{d+1}|_{\im(B_{d+1}^\intercal) }$ induces a bijection between $\im(B_{d+1}^\intercal)$ and $\im(B_{d+1})$. We see that 
\begin{align}
(B_{d+1}\fup d B_{d+1}^\intercal)(B_{d+1}) = B_{d+1}\fup d B_{d+1}^\intercal B_{d+1} = (B_{d+1})(P_{d+1}\fup d B_{d+1}^{\intercal}B_{d+1})\, ,
\end{align}
where in the last step we have used that $B_{d+1}P_{d+1} = B_{d+1}$, as $P_{d+1}$ is the projection along the kernel of $B_{d+1}$. This shows that  $B_{d+1}|_{\im(B_{d+1}^\intercal) }$ indeed conjugates the system $P_{d+1}\fup d B_{d+1}^{\intercal}B_{d+1}$ on $\im(B_{d+1}^\intercal)$ to the system $B_{d+1}\fup d B_{d+1}^\intercal$ on $\im(B_{d+1})$. \\
The statements on $\gdown d(\theta)$ follow in precisely the same way, where the role of the map $B_{d+1}$ is now played by $B_d^{\intercal}$.
\end{proof}

\begin{remk}\label{remarkonpfp2}
We may replace $\fup d$ in $\gup d = B_{d+1}\fup d B_{d+1}^\intercal$ by its canonical representative $P_{d+1}\fup d P_{d+1}$. Lemma \ref{conjugacieseverywhe} then tells us that the restriction of $\gup d$ to $\im(B_{d+1})$ is conjugate to  
\begin{align}
&\dot{x} = (B_{d+1}^{\intercal}B_{d+1})(P_{d+1}\fup d P_{d+1})(x)  \, \text{ and } \\ \nonumber
&\dot{y} = P_{d+1}^2\fup d P_{d+1}(B_{d+1}^{\intercal}B_{d+1}y) = (P_{d+1}\fup d P_{d+1})(B_{d+1}^{\intercal}B_{d+1})(y) \, .
\end{align}

\noindent As both $P_{d+1}$ and $B_{d+1}^{\intercal}$ map into $\im(B_{d+1}^\intercal)$, we may further write these two systems on $\im(B_{d+1}^\intercal)$ as
\[\dot{x} = (B_{d+1}^{\intercal}B_{d+1})_{\im(B_{d+1}^\intercal)}(P_{d+1}\fup d P_{d+1})(x)  \quad \text{ and } \quad \dot{y} =  (P_{d+1}\fup d P_{d+1})_{\im(B_{d+1}^\intercal)}(B_{d+1}^{\intercal}B_{d+1})(y) \, ,\]
where the two terms $(B_{d+1}^{\intercal}B_{d+1})$ and $(P_{d+1}\fup d P_{d+1})$ are now both seen as maps from $\im(B_{d+1}^\intercal)$ to itself. If we replace $\fup d$ by its class  under $\sim^{\uparrow}_{d+1}$ given by Equation \eqref{easierformqfq} for some 
 $f: \im(B_{d+1}^\intercal) \rightarrow \im(B_{d+1}^\intercal)$, then we simply get the systems
\[\dot{x} = \lup {d+1} f(x)  \quad \text{ and } \quad \dot{y} =  f(\lup {d+1} y) \, .\]
Here we have set 
\[ \lup {D}:= (B_{D}^{\intercal}B_{D})_{\im(B_{D}^\intercal)}: \im(B_{D}^\intercal) \rightarrow \im(B_{D}^\intercal) \, \text{ for all } D. \]
We claim that $\lup {D}$ is a symmetric, positive definite matrix. It is clear the matrix is indeed symmetric. Moreover, if $\lup {D}v = \lambda v$ for some $v \in \im(B_{D}^\intercal)$ and $\lambda \in \R$, then 
\begin{align}\label{calcwithnorm}
\lambda \|v\| = \langle \lambda v, v \rangle = \langle \lup {D}v, v \rangle =  \langle B_{D}^{\intercal}B_{D}v, v \rangle = \langle B_{D}v, B_{D}v \rangle = \| B_{D}v\| \geq 0\, ,
\end{align}
which shows that $\lambda$ is non-negative. If in fact $\lambda = 0$, then we have $v \in \ker(B_D)$ by Equation \eqref{calcwithnorm}. However, as we also have $v \in \im(B_{D}^\intercal)$, we conclude by Lemma \ref{spilttingseverywhere} that $v=0$. Hence, $\lup {D}$ is indeed positive definite. Heuristically, $\lup {D}$ should be thought of as the invertible part of the positive semi-definite matrix $B_{D}^{\intercal}B_{D}: C_D(X) \rightarrow C_D(X)$. \\

\noindent Completely analogues, we find that $\gdown d = B_{d}^\intercal\fdown d B_{d}$ restricted to $\im(B_{d}^\intercal)$ is conjugate to the systems
\[\dot{x} = \ldown d g(x) \quad \text{ and } \quad \dot{y} = g(\ldown d y) \, \text{ on } \im(B_{d}), \]
where $g: \im(B_{d}) \rightarrow \im(B_{d})$ is chosen so that $\fdown d$ is in the same class under $\sim^{\downarrow}_{d-1}$ as Equation  \eqref{easierformqfq2}, and where we have set
\[ \ldown {d}:= (B_d B_{d}^{\intercal})_{\im(B_{d})}: \im(B_{d}) \rightarrow \im(B_{d}) \, \text{ for all } d. \]
It follows again that $\ldown {d}$ is symmetric positive definite, and can be seen as the invertible part of $B_d B_{d}^{\intercal}: C_{d-1}(X) \rightarrow  C_{d-1}(X)$. \hfill $\triangle$
\end{remk}

\noindent The following result is directly motivated by remark \eqref{remarkonpfp2}. It relates the `input maps'   $f: \im(B_{d+1}^\intercal) \rightarrow \im(B_{d+1}^\intercal)$ and $g: \im(B_{d}) \rightarrow \im(B_{d})$ to the effective dynamics $\lup {d+1} f$ and  $\ldown d g$. Similar results apply to the conjugate vector fields $g \circ \ldown d$ and $f \circ \lup {d+1}$, but the former allow for a more straightforward formulation. Recall that a symmetric matrix has only real eigenvalues.

\begin{prop}\label{fixedpointsstabili}
Let $h: \R^n \rightarrow \R^n$ be a vector field and denote by $L \in \R^{n \times n}$  a symmetric, positive definite matrix. The fixed points of $h$ are the same as those of $Lh$. Moreover, if $x^*$ is a fixed point of $h$ and if $Dh(x^*)$ is symmetric, then $D(Lh)(x^*) = LDh(x^*)$ has only real eigenvalues. In that case the number of positive, zero and negative eigenvalues of $D(Lh)(x^*)$ is the same as the number of positive, zero and negative eigenvalues of $Dh(x^*)$. 
\end{prop}

\noindent The proof of Proposition \ref{fixedpointsstabili} uses \emph{Sylvester's law of inertia}~\cite[Theorem 4.5.8]{Horn.Johnson.book}. It states that for any symmetric matrix $A \in \R^{n \times n}$ and any invertible matrix $S \in \R^{n \times n}$, the symmetric matrices $A$ and $SAS^{\intercal}$ have the same number of positive, zero and negative eigenvalues. 

\begin{proof}[Proof of Proposition \ref{fixedpointsstabili}]
As $L$ is invertible, it is clear that the fixed points of $h$ are the same as those of $Lh$. From the fact that $L$ is positive definite, it follows that an invertible matrix $\sqrt{L}$ exists such that $\sqrt{L}^2 = L$.  Now if $Dh(x^*)$ is symmetric, it follows from Sylvester's law of inertia that  the symmetric matrix $\sqrt{L}Dh(x^*)\sqrt{L}$ has the same number of positive, zero and negative eigenvalues as $Dh(x^*)$. Hence, the matrix $LDh(x^*) = \sqrt{L}(\sqrt{L}Dh(x^*)\sqrt{L})\sqrt{L}^{-1}$ has only real eigenvalues, with the same number of positive, zero and negative ones as $Dh(x^*)$. This completes the proof.
\end{proof}

\begin{remk}
If $\fup d$ acts componentwise, then the Jacobian $D\fup d(\theta)$ is diagonal and therefore symmetric at any point $\theta \in C_d(X)$. Moreover, If $\fup d$ is equivalent under $\sim^{\uparrow}_d$ to Equation \eqref{easierformqfq} for some 
 $f: \im(B_{d+1}^\intercal) \rightarrow \im(B_{d+1}^\intercal)$, then with slight abuse of notation, we may write $f = P_{d+1} \fup d P_{d+1}$. It follows that the Jacobian of $f$ is likewise symmetric at any point in its phase space. We conclude that Proposition \ref{fixedpointsstabili} applies for componentwise $\fup d$. The analogous result of course holds when $\fdown d$ acts componentwise. \hfill $\triangle$
\end{remk}
\begin{cor}\label{cor:tree}
Suppose we have $\ker(B_{d+1}) = 0$. In other words, there are no non-zero $d+1$-cycles. Then, there is a bijection between the fixed points of $\gup d|_{\im(B_{d+1})}: \im(B_{d+1}) \rightarrow \im(B_{d+1})$ and  those of $\fup d$. If $\fup d$ is moreover componentwise, then this bijection preserves stability, i.e. the number of  positive, zero and negative eigenvalues of the corresponding linearized system. 
\end{cor}

\begin{remark}
  The results of Corollary~\ref{cor:tree} are well-known in the literature on consensus systems; basically any consensus system on a tree will converge, since there are no loops to break ``local convergence''.
\end{remark}

Finally, we address the question of when a system like~\eqref{eq:1'} can be written as the gradient of a potential system, as a condition on $\fup n$ and $\fdown n$.   

\begin{definition}
For any vector field $F\colon \R^n\to\R^n$ written in coordinates, we write $JF$ as the $n\times n$ matrix with components
\begin{equation*}
  \left( JF(x)\right)_{ij} = \left( \frac{\partial F_i}{\partial x_j}(x)\right)
\end{equation*}
\end{definition}

\begin{prop}
  The vector field $G_d(\theta_d)$ is exact if and only if the canonical representatives $P_{d+1}\fup d P_{d+1}$ and $Q_{d-1}\fdown d Q_{d-1}$ are exact.
\end{prop}

\begin{proof}
  Since $C_d(X)$ is contractible, we know that $G_d$ being closed implies that it is exact, so we check when $G_d$ is closed.  We are ultimately checking whether $JG_d$ is symmetric.
  
  Note first that $(J\gdown d-(J\gdown d)^\intercal) = B_d^\intercal (J\fdown d -(J\fdown d)^\intercal)  B_d,$ and by Remark~\ref{remarkonpfp2}, this is zero iff $Q_{d-1}(J\fdown d -(J\fdown d)^\intercal)Q_{d-1} = {\bf 0}$, and this is zero iff $Q_{d-1}(J\fdown d)  Q_{d-1}$ is symmetric.  Similarly, $J\gup d$ is symmetric iff $P_{d+1}(J\fup) d P_{d+1}$ is.  By the triple decomposition Lemma~\ref{lem:decomp}, $G_d$ is exact iff each of the two pieces $\gup d$ and $\gdown d$ are, and we are done.
\end{proof}

\section{Different orientations}\label{sec:orientation}

The signs in each map $B_d$ are determined by an orientation in the following way. We start by choosing a labelling of the vertices (i.e., the 0-simplices), after which we denote the $d$-simplex containing the vertices $v_{i_0}, v_{i_1}, \dots, v_{i_d} \in \{v_1, \dots, v_n\}$  by $[v_{i_0}, v_{i_1} \dots, v_{i_d}]$, with the convention that $i_0 < i_1 < \dots < i_d$. In this section we will simply denote the vertices by their label, so that the aforementioned $d$-simplex becomes  $[{i_0}, {i_1} \dots, {i_d}]$, still  with $i_0 < i_1 < \dots < i_d$. The boundary operator $\partial_d $ is then given by
\begin{align}\label{bound}
\partial_d  [i_0, i_1 \dots, i_d] = \sum_{k=0}^d (-1)^k [i_0, i_1, \dots, \widehat{i_k}, \dots, i_d]\, ,
\end{align}
where the symbol $\widehat{i_k}$  indicates this vertex is left out. Note that the vertices in $[i_0, i_1, \dots, \widehat{i_k}, \dots, i_d]$ are again in increasing order. \\

\noindent It is clear that the signs in Equation \eqref{bound} are a consequence of the original labelling of the vertices. Consider for instance the $2$-simplex $t = [1,2,3]$, together with its boundary consisting of the 1-simplices $e_1 = [2,3], e_2 = [1,3]$ and $e_3 = [1,2]$. Equation \eqref{bound} now reads
\begin{align}\label{bound2}
\partial_2 t = \partial_2[1,2,3] = [2,3] - [1,3] + [1,2] = e_1 - e_2 + e_3\, .
\end{align}
However, another labelling of the vertices in the complex might for instance give a renaming 
\begin{align}
1 \mapsto 7 \qquad
2 \mapsto 4 \qquad 
3 \mapsto 8\, ,
\end{align}
(assuming at least 8 vertices are present in the entire complex). Using our formalism of writing simplices with increasing numbers, we would now write $t = [4,7,8]$, $e_1 = [4,8]$, $e_2 = [7,8]$ and $e_3 = [4,7]$.  Equation \eqref{bound} then becomes
\begin{align}\label{bound3}
\partial_2 t = \partial_2[4,7,8] = [7,8] - [4,8] + [4,7] = e_2 - e_1 + e_3 = -e_1 + e_2 + e_3\, .
\end{align}
Equations \eqref{bound2} and \eqref{bound3} do not agree, from which we see that a different labelling can give another map $\partial_d$. In the classical case of the Laplacian on the space of $0$-simplices, these differences are known to drop out in the expression $L := B_2B_2^\intercal $ (or in general in $B_2FB_2^\intercal $ if $F$ has odd components). The following examples show this is not the case for general simplices.

\begin{figure}[ht]
\centering{
\begin{tikzpicture}

    \fill[fill=blue!45](0,0)--(1.5,2.6)--(3,0)--(1.5,-2.6)--(0,0);
	\node[circle,draw=black, fill=white, fill opacity = 1, inner sep=1.5pt, minimum size=14pt] (1) at (0,0) {$1$};
	\node[circle,draw=black, fill=white, fill opacity = 1, inner sep=1.5pt, minimum size=14pt] (3) at (1.5,2.6) {$2$};
	\node[circle,draw=black, fill=white, fill opacity = 1, inner sep=1.5pt, minimum size=14pt] (2) at (3,0) {$3$};
	\node[circle,draw=black, fill=white, fill opacity = 1, inner sep=1.5pt, minimum size=14pt] (4) at (1.5,-2.6) {$4$};
	\node[ fill=none, fill opacity = 1, inner sep=1.5pt, minimum size=14pt] (t1) at (1.5,1) {$t_1$};
	\node[ fill=none, fill opacity = 1, inner sep=1.5pt, minimum size=14pt] (t2) at (1.5,-1) {$t_2$};
	\node[ fill=none, fill opacity = 1, inner sep=1.5pt, minimum size=14pt] (e1) at (0.4,1.3) {$e_1$};
	\node[ fill=none, fill opacity = 1, inner sep=1.5pt, minimum size=14pt] (e2) at (2.6,1.3) {$e_2$};	
	\node[ fill=none, fill opacity = 1, inner sep=1.5pt, minimum size=14pt] (e3) at (1.5,0.2) {$e_3$};
	\node[ fill=none, fill opacity = 1, inner sep=1.5pt, minimum size=14pt] (e4) at (0.4,-1.3) {$e_4$};
	\node[ fill=none, fill opacity = 1, inner sep=1.5pt, minimum size=14pt] (e5) at (2.6,-1.3) {$e_5$};

	   \fill[fill=blue!45](5,0)--(6.5,2.6)--(8,0)--(6.5,-2.6)--(5,0);
	\node[circle,draw=black, fill=white, fill opacity = 1, inner sep=1.5pt, minimum size=14pt] (1s) at (5,0) {$1$};
	\node[circle,draw=black, fill=white, fill opacity = 1, inner sep=1.5pt, minimum size=14pt] (3s) at (6.5,2.6) {$3$};
	\node[circle,draw=black, fill=white, fill opacity = 1, inner sep=1.5pt, minimum size=14pt] (2s) at (8,0) {$4$};
	\node[circle,draw=black, fill=white, fill opacity = 1, inner sep=1.5pt, minimum size=14pt] (4s) at (6.5,-2.6) {$2$};
	\node[ fill=none, fill opacity = 1, inner sep=1.5pt, minimum size=14pt] (t1s) at (6.5,1) {$t_1$};
	\node[ fill=none, fill opacity = 1, inner sep=1.5pt, minimum size=14pt] (t2s) at (6.5,-1) {$t_2$};
	\node[ fill=none, fill opacity = 1, inner sep=1.5pt, minimum size=14pt] (e1s) at (5.4,1.3) {$e_1$};
	\node[ fill=none, fill opacity = 1, inner sep=1.5pt, minimum size=14pt] (e2s) at (7.6,1.3) {$e_2$};	
	\node[ fill=none, fill opacity = 1, inner sep=1.5pt, minimum size=14pt] (e3s) at (6.5,0.2) {$e_3$};
	\node[ fill=none, fill opacity = 1, inner sep=1.5pt, minimum size=14pt] (e4s) at (5.4,-1.3) {$e_4$};
	\node[ fill=none, fill opacity = 1, inner sep=1.5pt, minimum size=14pt] (e5s) at (7.6,-1.3) {$e_5$};

	\draw [ line width=0.7mm, >=stealth,  black] (1) to [bend right = 0] (2);
	\draw [ line width=0.7mm, >=stealth,  black] (1) to [bend right = 0] (3);
	\draw [ line width=0.7mm, >=stealth,  black] (1) to [bend right = 0] (4);
	\draw [ line width=0.7mm, >=stealth,  black] (2) to [bend right = 0] (3);
	\draw [ line width=0.7mm, >=stealth,  black] (2) to [bend right = 0] (4);

	\draw [ line width=0.7mm, >=stealth,  black] (1s) to [bend right = 0] (2s);
	\draw [ line width=0.7mm, >=stealth,  black] (1s) to [bend right = 0] (3s);
	\draw [ line width=0.7mm, >=stealth,  black] (1s) to [bend right = 0] (4s);
	\draw [ line width=0.7mm, >=stealth,  black] (2s) to [bend right = 0] (3s);
	\draw [ line width=0.7mm, >=stealth,  black] (2s) to [bend right = 0] (4s);
\end{tikzpicture}
\caption{Two copies of the same simplicial complex, but with a different labelling of the vertices. This seemingly innocuous difference is enough to give different dynamical systems.}
\label{fig_label}}
\end{figure}
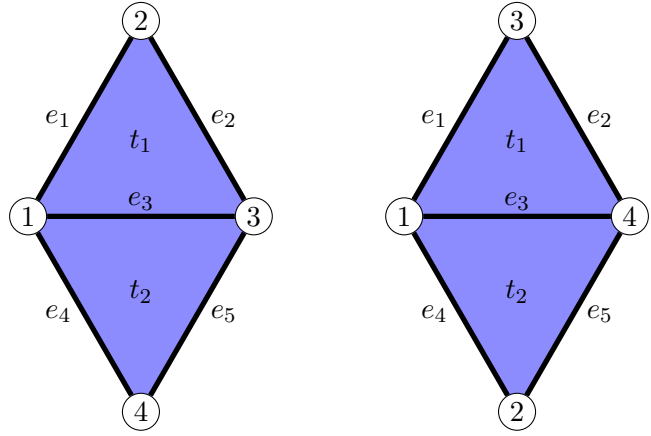

\begin{ex}\label{examlabeleffect}
We consider the simplicial complex of Figure \ref{fig_label}, where we focus on the dynamics of the edges communicating through the two triangles. We start by looking at the labelling of the vertices on the left. This gives us
\begin{align}
\partial_2(t_1) &= \partial_2 [1,2,3] = [2,3] - [1,3] + [1,2] = e_2 - e_3 + e_1 \\ \nonumber
\partial_2(t_2) &= \partial_2 [1,3,4] = [3,4] - [1,4] + [1,3] = e_5 - e_4 + e_3\, .
\end{align}
Expressed in the bases $\{t_1, t_2\}$ and $\{e_1, e_2, e_3, e_4, e_5\}$, we therefore find 
\begin{equation}
B_2 = \begin{pmatrix}
1 & 0 \\
1 & 0 \\
-1 & 1 \\
0 & -1 \\
0 & 1 \\
\end{pmatrix}\, .
\end{equation}
Let us assume $\fup1: C^2(X) \rightarrow C^2(X)$ is componentwise, given by $\fup1(x_{t_1}, x_{t_2}) = (f(x_{t_1}), f(x_{t_2}))^\intercal $ for some odd function $f: \R \rightarrow \R$. We obtain the vector field
\begin{align}\label{leftODEtriangles}
B_2\fup1B^\intercal _2(x) &= \left(\begin{array}{l}
f(x_1 + x_2 - x_3) \\
f(x_1 + x_2 - x_3)  \\
-f(x_1 + x_2 - x_3) + f(x_3 - x_4 + x_5) \\
-f(x_3 - x_4 + x_5) \\
f(x_3 - x_4 + x_5) \\
\end{array} \right) \\ \nonumber
&= \left(\begin{array}{l}
f(x_1 + x_2 - x_3) \\
f(x_2 + x_1 - x_3)  \\
f(x_3 - x_1 - x_2 ) + f(x_3 - x_4 + x_5) \\
f(x_4 -x_3 - x_5) \\
f(x_5 + x_3 - x_4 ) \\
\end{array} \right) \, .
\end{align}
Here we write $x_1 := x_{e_1}$ for the coordinate in $e_1$, and so forth. If we instead look at the labelling of the vertices on the right of Figure \ref{fig_label}, we obtain
\begin{align}
\partial_2(t_1) &= \partial_2 [1,3,4] = [3,4] - [1,4] + [1,3] = e_2 - e_3 + e_1 \\ \nonumber
\partial_2(t_2) &= \partial_2 [1,2,4] = [2,4] - [1,4] + [1,2] = e_5 - e_3 + e_4\, .
\end{align}
Hence, we instead find
\begin{equation}
B_2 = \begin{pmatrix}
1 & 0 \\
1 & 0 \\
-1 & -1 \\
0 & 1 \\
0 & 1 \\
\end{pmatrix}\, .
\end{equation}
As a result, our vector field becomes
\begin{align}\label{rightODEtriangles}
B_2\fup1B^\intercal _2(x) &= \left(\begin{array}{l}
f(x_1 + x_2 - x_3) \\
f(x_1 + x_2 - x_3)  \\
-f(x_1 + x_2 - x_3) - f(-x_3 + x_4 + x_5) \\
f(-x_3 + x_4 + x_5) \\
f(-x_3 + x_4 + x_5) \\
\end{array} \right) \\ \nonumber
&= \left(\begin{array}{l}
f(x_1 + x_2 - x_3) \\
f(x_2 + x_1 - x_3)   \\
f(x_3 - x_1 - x_2) + f(x_3 - x_4 - x_5) \\
f(x_4 -x_3 + x_5)\\
f(x_5 -x_3 + x_4) \\
\end{array} \right) \, .
\end{align}
Comparing equations \eqref{leftODEtriangles} and \eqref{rightODEtriangles}, we see that they are not the same. In particular, note that the vector field of Equation \eqref{rightODEtriangles} has the invariant space $\{x_4 = x_5\}$, for any choice of $f$, whereas for a generic choice of (odd) function $f$ this is not an invariant space for the vector field of Equation \eqref{leftODEtriangles}.  \hfill $\triangle$
\end{ex}

\begin{ex}\label{examlabeleffect2}
We return to Figure \ref{fig_label} , where we instead look at the dynamics of the triangles, communicating through their mutually adjacent edges. We again assume $\fdown2$ to be componentwise, with components given by an odd function $f:\R \rightarrow \R$. The labelling on the left of Figure \ref{fig_label} gives
\begin{align}
B_2^\intercal \fdown2B_2(x) &= \left(\begin{array}{l}
f(x_1) + f(x_1)- f(-x_1+x_2) \\ \nonumber
f(-x_1+x_2) - f(-x_2) + f(x_2) 
\end{array} \right) \\ \nonumber
&= \left( \begin{array}{l}
2f(x_1) +  f(x_1 - x_2) \\ \nonumber
2f(x_2) + f(x_2 - x_1) 
\end{array}  \right) \, ,
\end{align}
where $x_i$ denotes the dynamics of triangle $t_i$, $i \in \{1,2\}$.  The labelling on the right of Figure \ref{fig_label} gives
\begin{align}
B_2^\intercal \fdown2B_2(x)  &= \left(\begin{array}{l}
f(x_1) + f(x_1)- f(-x_1-x_2) \\ \nonumber
-f(-x_1-x_2) + f(x_2) + f(x_2) 
\end{array} \right) \\ \nonumber
&= \left( \begin{array}{l}
2f(x_1) +  f(x_1 + x_2) \\ \nonumber
2f(x_2) + f(x_2 + x_1) 
\end{array}  \right) \, .
\end{align}
Again, we find different vector fields. \hfill $\triangle$
\end{ex}

\noindent Even though examples \ref{examlabeleffect} and \ref{examlabeleffect2} show that the vector fields $ \gup d$ and $ \gdown d$ may depend on the labelling of the vertices, we will see in Theorem \ref{mainthronlabel} below that different choices in fact give conjugate systems. To this end, we first define:

\begin{defi}\label{signdefff}
Let $i_0<i_1< \dots<i_d \in \{1, \dots, n\}$ be $d+1$ distinct numbers, and let $\sigma:  \{1, \dots, n\} \rightarrow  \{1, \dots, n\}$ be a permutation. The sign of the set $A = \{i_0, \dots, i_d\}$ with respect to $\sigma$ is the number $\sgn(A, \sigma) \in \{1, -1\}$ defined as follows: The numbers $i_0$ through $i_d$ are in ascending order, but the numbers $\sigma(i_0)$ through $\sigma(i_d)$ need not be. However, there exists a unique permutation $\tau:   \{0, \dots, d\} \rightarrow  \{0, \dots, d\}$ such that $\sigma(i_{\tau(0)})<\sigma(i_{\tau(1)})< \dots<\sigma(i_{\tau(d)})$. We then simply set $\sgn(A, \sigma) := \sgn(\tau)$. In other words, the number $\sgn(A, \sigma)$ tells us if $\sigma$ rearranges the elements of $A$ as an even or an odd permutation. \\
Next, suppose we have fixed a labelling on the vertices of a simplicial complex $X$. For every value of $d$ we define a linear map $T_{\sigma}^d : C_d(X) \rightarrow C_d(X)$ by simply setting $(T_{\sigma}^d(x))_s = \sgn(A_s, \sigma)x_s$, where $A_s = \{i_0, \dots, i_d\}$ if $s$ is the $d$-simplex $s = [i_0, i_1, \dots, i_d]$. Note that $T_{\sigma}^d$ is a diagonal map with only $1$s and $-1$s on the diagonal. In particular, we see that $T_{\sigma}^d$ equals its own inverse. \hfill $\triangle$
\end{defi} 

\begin{remk}
It is well-known that the sign of a permutation $\tau$ equals $1$ if there is an even number of pairs $(k,l) \in \{1, \dots, n\}^2$ such that $k>l$ but $\tau(k) < \tau(l)$, and $-1$ if there is an odd number of such pairs. Suppose $\tau$ is determined by $\sigma$ and $A$ as in the definition of $\sgn(A, \sigma)$ above. Then for $k, l \in \{0, \dots, d\}$ we see that $k > l$ if and only if $\sigma(i_{\tau(k)})>\sigma(i_{\tau(l)})$. On the other hand, we have  $\tau(k) < \tau(l)$ if and only if $i_{\tau(k)} < i_{\tau(l)}$. Hence, we see that $ \sgn(\tau) = \sgn(A, \sigma) = 1$ precisely when there is an even number of pairs $(i,j) \in A^2$ such that $i<j$ but $\sigma(i) > \sigma(j)$, and $-1$ otherwise. \hfill $\triangle$
\end{remk}

\begin{ex}\label{Tsigmas}
We return to the complex $X$ of Figure \ref{fig_label}. The labelling on the right is obtained from the  one on the left by a permutation $\sigma$, defined by
\[\sigma(1) = 1\quad \sigma(2) = 3\quad \sigma(3) = 4\quad \sigma(4) = 2\, .\]
The numbers of the simplex $t_1$ are acted on by $\sigma$ as
\[(1,2,3) \mapsto (1,3,4) \, .\]
As the entries of $(1,3,4)$ are again in increasing order, we conclude that $\sgn(A_{t_1}, \sigma) = 1$. Likewise, for $t_2$ we find 
\[(1,3,4) \mapsto (1,4,2) \, .\]
It requires one transposition (switching $2$ and $4$) to bring $(1,4,2)$ into increasing order, and so we conclude that $\sgn(A_{t_2}, \sigma) = -1$. With regards to the ordered basis $t_1, t_2$ for $C^2(X)$ we therefore find
\begin{equation}
T_{\sigma}^2 = \begin{pmatrix}
1 & 0 \\
0 & -1 \\
\end{pmatrix}\, .
\end{equation}
Likewise, $\sigma$ respects the ordering of the edges $e_1, e_2, e_3$ and $e_4$, whereas it switches the ordering of $e_5$. With regards to the ordered basis $e_1, \dots, e_5$ we therefore find 
\begin{equation}
T_{\sigma}^1 = \begin{pmatrix}
1 & 0 & 0 & 0 & 0\\
0 & 1 & 0 & 0 & 0\\
0 & 0 & 1 & 0 & 0\\
0 & 0 & 0 & 1 & 0\\
0 & 0 & 0 & 0 & -1\\
\end{pmatrix}\, .
\end{equation}
Note that both matrices indeed square to the identity. \hfill $\triangle$
\end{ex}

\begin{thr}\label{mainthronlabel}
Let $X$ be a simplicial complex with a fixed labelling of the vertices. Suppose another labelling is given, and let $\sigma:  \{1, \dots, n\} \rightarrow  \{1, \dots, n\}$ be the permutation such that vertex $i$ in the first labelling equals vertex $\sigma(i)$ in the second labelling. Denote by $B_d$ and $\tilde{B}_d$ the linear boundary maps induced by the first and second labelling, respectively. Then we have the conjugacy relations 
\begin{align}
T_{\sigma}^{d-1} \circ B_d\fup {d-1}B_d^\intercal  &= \tilde{B}_d\fup {d-1}\tilde{B}_d^\intercal  \circ T_{\sigma}^{d-1} \, ;\\ \nonumber
T_{\sigma}^{d} \circ B_d^\intercal \fdown dB_d &= \tilde{B}_d^\intercal \fdown d \tilde{B}_d \circ T_{\sigma}^{d} \, ,
\end{align}
for all $\fup {d-1}: C_d(X) \rightarrow C_d(X)$ and $\fdown {d}: C_{d-1}(X) \rightarrow C_{d-1}(X)$ that are componentwise with odd components. In particular, a different choice of labelling gives conjugate vector fields.
\end{thr}
The following lemma is the key step in proving Theorem \ref{mainthronlabel}.
\begin{lem} \label{lemforlabel}
Let $X$ be a simplicial complex and suppose  $\sigma$, $B_d$ and $\tilde{B}_d$ are as in Theorem \ref{mainthronlabel}. We have
\[ B_d \circ T_{\sigma}^d = T_{\sigma}^{d-1} \circ \tilde{B}_d \, .\]
\end{lem}

\begin{proof}
We start by choosing a $d$-simplex $e$. In the first labelling, we write $e = [i_0, i_1, \dots, i_d]$ with $i_0 < i_1 < \dots<i_d \in \{1, \dots, n\}$. In the second labelling we write $e = [j_0, j_1, \dots, j_d]$ with $j_0 < j_1 < \dots<j_d \in \{1, \dots, n\}$. The permutation $\sigma$ is defined so that a vertex $i_k$ in the first labelling is called $\sigma(i_k)$ in the second labelling. In particular, since $e$ contains a fixed set of vertices, we necessarily have 
\[\{j_0, \dots, j_d\} = \{\sigma(i_0), \dots, \sigma(i_d)\} \]
as sets. It follows that there is some bijection $\tau: \{0, \dots, d\} \rightarrow \{0, \dots, d\}$ such that $j_k = \sigma(i_{\tau(k)})$ for all $k \in \{0, \dots, d\}$. In other words, the sequence $\sigma(i_{\tau(0)}), \sigma(i_{\tau(1)}) \dots, \sigma(i_{\tau(d)})$ is increasing by definition of $\tau$. \\
Using this notation, we proceed to compare $B_d \circ T_{\sigma}^d$ to $T_{\sigma}^{d-1} \circ \tilde{B}_d$. Note that, since $T_{\sigma}^{d-1}$ and $T_{\sigma}^{d}$ both square to the identity, it suffices to show that $T_{\sigma}^{d-1} \circ B_d \circ T_{\sigma}^d = \tilde{B}_d$. For simplicity, we will identify simplices with their corresponding basis elements in the vector spaces $C^d(X)$ and $C^{d-1}(X)$. \\
We start with $T_{\sigma}^{d-1} \circ B_d \circ T_{\sigma}^d$. For the $d$-simplex $e$ we get
\begin{align}\label{signsofi}
T_{\sigma}^{d-1}  B_d  T_{\sigma}^d(e) &= T_{\sigma}^{d-1}  B_d \sgn(A_e, \sigma) e \\ \nonumber
&= \sgn(A_e, \sigma)T_{\sigma}^{d-1}  \partial_d [i_0, i_1 \dots, i_d]  \\ \nonumber
&= \sgn(A_e, \sigma)T_{\sigma}^{d-1}\sum_{k=0}^d (-1)^k [i_0, i_1, \dots, \widehat{i_k}, \dots, i_d]  \\ \nonumber
&= \sum_{k=0}^d\sgn(A_e, \sigma)(-1)^k\sgn(A_e\setminus \{i_k\}, \sigma) [i_0, i_1, \dots, \widehat{i_k}, \dots, i_d] \, .
\end{align}
On the other hand, we have
\begin{align}\label{signsofj}
 \tilde{B}_d(e) &=   \partial_d [j_0, j_1 \dots, j_d] = \sum_{k=0}^d (-1)^k [j_0, j_1, \dots, \widehat{j_k}, \dots, j_d]  \, .
\end{align}
Now, recall that vertex $i_k$ in the first labelling equals vertex $\sigma(i_k)$ in the second labelling. Moreover, we have defined $\tau$ such that $j_l = \sigma(i_{\tau(l)})$ for all $l \in \{0, \dots, d\}$. It follows that $j_{\tau^{-1}(k)} = \sigma(i_{\tau(\tau^{-1}(k))}) = \sigma(i_k)$. Hence, vertex $i_k$, using the first labelling, is the same as vertex $j_{\tau^{-1}(k)}$ in the second labelling. In particular, the $(d-1)$-simplex $[i_0, i_1, \dots, \widehat{i_k}, \dots, i_d]$ is the same object as $[j_0, j_1, \dots, \widehat{j_{\tau^{-1}(k)}}, \dots, j_d]$. In Equation \eqref{signsofi} this simplex is given the sign $\sgn(A_e, \sigma)(-1)^k\sgn(A_e\setminus \{i_k\}, \sigma)$, whereas in Equation \eqref{signsofj} it is given the sign $(-1)^{{\tau^{-1}(k)}}$. This means the proof is done once we show that 
\begin{equation}\label{usefulidentitytau}
 \sgn(A_e, \sigma)(-1)^k\sgn(A_e\setminus \{i_k\}, \sigma)  = (-1)^{{\tau^{-1}(k)}}\, .
 \end{equation}
To this end, note that we wish to switch elements so that $(\sigma(i_0), \dots, \sigma(i_d))$ becomes \break $(\sigma(i_{\tau(0)}), \dots, \sigma(i_{\tau(d)}))$. If this can be done by switching $m$ pairs, then by definition $\sgn(A_e, \sigma) = (-1)^m$. We first bring $\sigma(i_k)$ to the right spot, without changing the ordering of the other elements $\sigma(i_l)$, $l\not=k$, among themselves. To this end, note that the identity $j_{\tau^{-1}(k)} = \sigma(i_k)$ simply says that $\sigma(i_k)$ has to end up on place $\tau^{-1}(k)$. Assuming $k \leq \tau^{-1}(k)$ we make the following $\tau^{-1}(k)- k$ switches:
\begin{align*}
&(\sigma(i_0), \dots, \sigma(i_{k-1}), \textcolor{red}{\sigma(i_k), \sigma(i_{k+1})}, \sigma(i_{k+2}), \dots, \sigma(i_{\tau^{-1}(k)}), \dots, \sigma(i_d)) \\ \nonumber
\mapsto &(\sigma(i_0), \dots, \sigma(i_{k-1}), \sigma(i_{k+1}),\textcolor{red}{\sigma(i_{k}), \sigma(i_{k+2})}, \dots, \sigma(i_{\tau^{-1}(k)}), \dots, \sigma(i_d)) \\ \nonumber
\mapsto &(\sigma(i_0), \dots, \sigma(i_{k-1}), \sigma(i_{k+1}),\sigma(i_{k+2}), \sigma(i_{k}), \dots, \sigma(i_{\tau^{-1}(k)}), \dots, \sigma(i_d)) \\ \nonumber
&\vdots \\ \nonumber
\mapsto &(\sigma(i_0), \dots, \sigma(i_{k-1}), \sigma(i_{k+1}), \dots, \sigma(i_{\tau^{-1}(k)}),\sigma(i_{k}), \sigma(i_{\tau^{-1}(k)+1}) \dots, \sigma(i_d))\, .
\end{align*}
If instead we have $k > \tau^{-1}(k)$ then similarly we may put $\sigma(i_k)$ in the right spot by using $k - \tau^{-1}(k)$ switches, so that in general we need $|k - \tau^{-1}(k)|$ of them. It remains to put all other elements $\sigma(i_l)$, $l\not=k$, in the right spot. As we may `jump over' the term $\sigma(i_k)$, this is equivalent to reordering the sequence $(\sigma(i_0), \dots, \widehat{\sigma(i_k)}, \dots, \sigma(i_d))$ till the elements are in ascending order. Hence, if this can be done in $p$ switches then by definition $\sgn(A_e\setminus \{i_k\}, \sigma)  = (-1)^p$. It takes $m$ switches in total to order $(\sigma(i_0), \dots, \sigma(i_d))$, so that we find
\begin{equation}
p + |k - \tau^{-1}(k)| = m \quad  \text{ modulo } 2\, ,
\end{equation}
where
\begin{equation}
\sgn(A_e, \sigma)  = (-1)^m \text{ and } \sgn(A_e\setminus \{i_k\}, \sigma)  = (-1)^p \, .
\end{equation}
Using the identity $(-1)^q = (-1)^{-q}$ for all $q \in \mathbb{Z}$, we obtain
\begin{align}
\sgn(A_e, \sigma)  &= (-1)^m = (-1)^{p + |k - \tau^{-1}(k)|} = (-1)^p(-1)^{k - \tau^{-1}(k)} \\ \nonumber
&= \sgn(A_e\setminus \{i_k\}, \sigma) (-1)^{k}(-1)^{-\tau^{-1}(k)}\, .
\end{align}
Reordering, and using that $(\sgn(A_e, \sigma))^2 = 1$, we finally obtain
\begin{align}
\sgn(A_e, \sigma) (-1)^k\sgn(A_e\setminus \{i_k\}, \sigma)  = (-1)^{\tau^{-1}(k)}\, .
\end{align}
This completes the proof.
\end{proof}

\begin{proof}[Proof of Theorem \ref{mainthronlabel}]
We start by remarking that $(T_{\sigma}^d)^{\intercal} = T_{\sigma}^d$ for all $d$. Moreover, as the components of the map $\fup {d-1}$ are assumed odd, we see that $T_{\sigma}^d$ commutes with $\fup {d-1}$. Finally, as we have $(T_{\sigma}^d)^2 = \Id_{C^d(X)}$, we see from Lemma \ref{lemforlabel} that we have both
\begin{align}
 T_{\sigma}^{d-1} \circ  B_d =  \tilde{B}_d \circ T_{\sigma}^d  \quad  \text{ and }  \quad     T_{\sigma}^d \circ B_d^{\intercal} = \tilde{B}_d^{\intercal}  \circ T_{\sigma}^{d-1}    \, .
\end{align}
 The proof now follows from an easy calculation. For instance
\begin{align}
T_{\sigma}^{d-1}  B_d\fup {d-1}B_d^\intercal  &=  \tilde{B}_d T_{\sigma}^{d}\fup {d-1}B_d^\intercal  =  \tilde{B}_d \fup {d-1}T_{\sigma}^{d}B_d^\intercal     = \tilde{B}_d \fup {d-1}\tilde{B}_d^\intercal  T_{\sigma}^{d-1} \, .
\end{align}
The other identity in Theorem  \ref{mainthronlabel} follows in a similar way, which completes the proof.
\end{proof}

\begin{ex}\label{thecalculationlabel}
Returning to examples \ref{examlabeleffect} and \ref{examlabeleffect2}, let us denote by $B_2$ the linear boundary map on $2$-simplices (triangles) induced by the labelling on the left of Figure \ref{fig_label}, and by $\tilde{B}_2$ the map induced by the labelling on the right of Figure \ref{fig_label}. We have calculated $T_{\sigma}^1$ and $T_{\sigma}^2$ in Example \ref{Tsigmas}. One easily verifies that indeed $T_{\sigma}^1 B_2\fup1B^\intercal _2 T_{\sigma}^1 = \tilde{B}_2\fup1\tilde{B}^\intercal _2$ and 
$T_{\sigma}^2 B_2^\intercal \fdown2B_2 T_{\sigma}^2 = \tilde{B}_2^\intercal  \fdown2 \tilde{B}_2$. \hfill $\triangle$
\end{ex}

\begin{remk}
The map $T_{\sigma}^{0}$ is readily seen to equal the identity on $C^0(X)$ for any permutation $\sigma$. This is because a $0$-simplex $[i]$ is sent to $[\sigma(i)]$ under $\sigma$, which is trivially already in `ascending order'. Hence, we retrieve the well-known fact that the classical Laplacian for vertices is independent of the orientation on the edges.  \hfill $\triangle$ \end{remk}

\begin{remk}
The proof of Theorem \ref{mainthronlabel} shows that this result does not require $\fup {d-1}$ and $\fdown {d}$ to be componentwise. For instance, it suffices if both maps commute with all diagonal matrices with diagonal entries in $\{1,-1\}$.  For $\fup {d-1}: C_d(X) \rightarrow C_d(X)$ this means that each component  $(\fup {d-1})_e$ for $e \in X_d$ is odd in the variable $x_e$ and even in the other variables $x_{f}$, $f \in X_d\setminus\{e\}$, with an analogous condition on  $\fdown {d}$. \hfill $\triangle$
\end{remk}

\section{Symmetries of simplicial dynamics}\label{sec:symmetries}
In this section we briefly discuss how symmetries of a simplicial complex $X$ induce symmetries of the corresponding dynamical systems $\gdown d = B_{d}^\intercal\fdown d B_{d}$ and  $\gup d = B_{d+1}\fup d B_{d+1}^\intercal$. Recall that we only consider finite complexes, and we furthermore make the assumption that each simplex is uniquely determined by the vertices it contains.  We start with a definition.

\begin{defi}\label{defisymm}
Let $X$ be a simplicial complex. A \emph{symmetry} of $X$ is a permutation $\sigma$ of the vertices of $X$ such that for any simplex $s = [v_0, \dots, v_d]$ there exists a simplex $t = [w_0, \dots, w_d]$ with $\{w_0, \dots, w_d\} = \{\sigma(v_0), \dots, \sigma(v_d)\}$.  If the simplices $s$ and $t$ are related to each other by $\sigma$ in this way, then we simply write $t = \sigma(s)$. Note that the symmetries of $X$ form a group, which we denote by $\mathcal{S}_X$. \\
Next,  assume we have fixed a labelling of the vertices of $X$, so that we will henceforth identify vertices with their labels. Given a symmetry $\sigma \in \mathcal{S}_X$, we may define two linear maps $\tilde{S}^d_{\sigma}, {S}^d_{\sigma}: C_d(X) \rightarrow C_d(X)$ for each degree $d$. The first one is the most straightforward: if we identify the elements of $X_d$ with the canonical basis for $C_d(X)$, then $\tilde{S}^d_{\sigma}$ is given on this basis by
\begin{equation}
\tilde{S}^d_{\sigma}(s) =  \sigma(s)\quad \text{ for all } s \in X_d \, .
\end{equation}
As for the second one, we have 
\begin{equation}
{S}^d_{\sigma}(s) =  \sgn(A_s, \sigma)\sigma(s) \quad \text{ for all } s \in X_d \, .
\end{equation}
Here $ \sgn(A_s, \sigma)$ is the sign of the set $A_s = \{i_0, \dots, i_d\}$ with respect to $\sigma$ as defined in Definition \ref{signdefff}, where $s = [i_0, \dots, i_d]$. \hfill $\triangle$
\end{defi}
The following lemma will be very useful for dealing with the term $\sgn(A_s, \sigma)$.

\begin{lem}\label{lemmaonsgn}
Let $s = [i_0, \dots, i_d]$ be a simplex and write $A_s = \{i_0, \dots, i_d\}$ for its corresponding set of vertices. Given two symmetries $\sigma, \tau \in \mathcal{S}_X$, we have
\begin{equation}
\sgn(A_{\sigma(s)}, \tau)\sgn(A_s, \sigma)  = \sgn(A_s, \tau\sigma)\, .
\end{equation}
\end{lem}

\begin{proof}
As we write $s = [i_0, \dots, i_d]$, we assume that $i_0 < \dots <i_d$. Suppose $\kappa_1$ is the permutation of $\{0, \dots, d\}$ that brings the sequence $({\sigma(i_0)}, \dots, {\sigma(i_d)})$ into ascending order. That is, we have 
\begin{equation}\label{kappa1inis}
{\sigma(i_{\kappa_1(0)})} < \dots < \sigma(i_{\kappa_1(d)})\, .
\end{equation}
Likewise, denote by $\kappa_2$ the permutation that brings the sequence $({\tau(j_0)}, \dots, {\tau(j_d)})$ into ascending order, where $A_{\sigma(s)} = \{j_0, \dots, j_d\}$ with $j_0 < \dots < j_d$. In other words, 
\begin{equation}\label{kappa2injs}
{\tau(j_{\kappa_2(0)})} < \dots < \tau(j_{\kappa_2(d)})\, .
\end{equation} 
From Equation \eqref{kappa1inis} we see that ${\sigma(i_{\kappa_1(0)})} = j_0$, ${\sigma(i_{\kappa_1(1)})} = j_1$ and so forth. Hence, from Equation \eqref{kappa2injs} we see that
\begin{equation}
\tau(j_{\kappa_2(0)}) = \tau({\sigma(i_{\kappa_1(\kappa_2(0))})})  =  \tau\sigma(i_{\kappa_1\kappa_2(0)}) < \dots < \tau(j_{\kappa_2(d)}) =    \tau\sigma(i_{\kappa_1\kappa_2(d)}) \, .
\end{equation} 
In other words, $\kappa_1\kappa_2$ rearranges the elements of $A_{\tau\sigma(s)}$ into ascending order. We therefore find
\begin{equation}
\sgn(A_{\sigma(s)}, \tau)\sgn(A_s, \sigma)  = \sgn(\kappa_2)\sgn(\kappa_1) = \sgn(\kappa_1\kappa_2) = \sgn(A_s, \tau\sigma)\, ,
\end{equation}
which completes the proof.
\end{proof}

Both linear maps of Definition \ref{defisymm} are natural candidates for symmetries of the systems $\gdown d$ and $\gup d$, in the following way:

\begin{lem}\label{lemonrepresess}
Both $\tilde{S}^d_{\sigma}$ and ${S}^d_{\sigma}$ define representations of $\mathcal{S}_X$. That is, we have $\tilde{S}^d_{\Id_{}} = {S}^d_{\Id_{}} = \id_{C_d(X)}$ and $\tilde{S}^d_{\sigma} \tilde{S}^d_{\tau}  = \tilde{S}^d_{\sigma\tau}$, ${S}^d_{\sigma} {S}^d_{\tau}  = {S}^d_{\sigma\tau}$ for all $\sigma, \tau \in \mathcal{S}_X$.
\end{lem}

\begin{proof}
It is clear from the definitions that $\tilde{S}^d_{\Id_{}} = {S}^d_{\Id_{}} = \id_{C_d(X)}$. Given $\sigma, \tau \in \mathcal{S}_X$ and $s \in X_d$ (seen as an element of $C_d(X)$), we have
\begin{align}
\tilde{S}^d_{\tau} \tilde{S}^d_{\sigma}(s) =    \tilde{S}^d_{\tau}(\sigma(s))  =  \tau\sigma(s) = \tilde{S}^d_{\tau\sigma}(s) \, .
\end{align} 
As $\tilde{S}^d_{\tau} \tilde{S}^d_{\sigma}$ and $\tilde{S}^d_{\tau\sigma}$ agree on a basis of $C_d(X)$, they agree as linear maps. Likewise, we find 
\begin{align}\label{calcrepresnottilde}
{S}^d_{\tau} {S}^d_{\sigma}(s) &= {S}^d_{\tau} \sgn(A_s, \sigma)\sigma(s) =\sgn(A_s, \sigma) {S}^d_{\tau} \sigma(s) = \sgn(A_s, \sigma)\sgn(A_{\sigma(s)}, \tau)\tau\sigma(s) \\ \nonumber
&= \sgn(A_s, \tau\sigma)\tau\sigma(s) = {S}^d_{\tau\sigma} (s)\, ,
\end{align}
where in the second-to-last step we have used Lemma \ref{lemmaonsgn}. This completes the proof. 
\end{proof}

Nevertheless, only one of these representations is guaranteed to induce a symmetry of the corresponding dynamical systems:

\begin{ex}\label{firstexamwithsdsig}
We return to the complex given by Figure \ref{fig_label}, with the vertex-labelling on the left. A symmetry of this simplex is given by the transposition of vertices $1$ and $3$, $\sigma = (1,3)$. This induces a permutation of the edges where $e_1$ and $e_2$ are switched, as are  $e_4$ and $e_5$. Moreover, the orientations of $e_1$ and $e_2$ are reversed, and so is the orientation of $e_3$. We therefore find
\begin{equation}
\tilde{S}_{\sigma}^1 = \begin{pmatrix}
0 & 1 & 0 & 0 & 0\\
1 & 0 & 0 & 0 & 0\\
0 & 0 & 1 & 0 & 0\\
0 & 0 & 0 & 0 & 1\\
0 & 0 & 0 & 1 & 0\\
\end{pmatrix} \, \text{ and }
{S}_{\sigma}^1 = \begin{pmatrix}
0 & -1 & 0 & 0 & 0\\
-1 & 0 & 0 & 0 & 0\\
0 & 0 & -1 & 0 & 0\\
0 & 0 & 0 & 0 & 1\\
0 & 0 & 0 & 1 & 0\\
\end{pmatrix}\, .
\end{equation}
In Example \ref{examlabeleffect} we have calculated $  \gup 1 = B_2\fup1B^\intercal _2$. We see that 
\begin{align}\label{leftODEtriangles2}
 \gup 1 (\tilde{S}_{\sigma}^1x) 
&= \left(\begin{array}{l}
f(x_2 + x_1 - x_3) \\
f(x_1 + x_2 - x_3)  \\
f(x_3 - x_2 - x_1 ) + f(x_3 - x_5 + x_4) \\
f(x_5 -x_3 - x_4) \\
f(x_4 + x_3 - x_5 ) \\
\end{array} \right) \text{whereas} \\ \nonumber
\tilde{S}_{\sigma}^1( \gup 1 (x) )
&=  \left(\begin{array}{l}
f(x_2 + x_1 - x_3)  \\
f(x_1 + x_2 - x_3) \\
f(x_3 - x_1 - x_2 ) + f(x_3 - x_4 + x_5) \\
f(x_5 + x_3 - x_4 ) \\
f(x_4 -x_3 - x_5) \\
\end{array} \right)\, .
\end{align}
These do not agree, and so the map $\tilde{S}_{\sigma}^1$ is not a symmetry of $\gup 1$. On the other hand, we have 
\begin{align}\label{leftODEtriangles3}
 \gup 1 ({S}_{\sigma}^1x) 
&=  \left(\begin{array}{l}
f(-x_2 - x_1 + x_3) \\
f(-x_1 - x_2 + x_3)  \\
f(-x_3 + x_2 + x_1 ) + f(-x_3 - x_5 + x_4) \\
f(x_5 +x_3 - x_4) \\
f(x_4 - x_3 - x_5 ) \\
\end{array} \right) \text{and} \\ \nonumber
{S}_{\sigma}^1( \gup 1 (x) )
&=  \left(\begin{array}{l}
-f(x_2 + x_1 - x_3)  \\
-f(x_1 + x_2 - x_3) \\
-f(x_3 - x_1 - x_2 ) - f(x_3 - x_4 + x_5) \\
f(x_5 + x_3 - x_4 ) \\
f(x_4 -x_3 - x_5) \\
\end{array} \right)\, .
\end{align}
Using the fact that $f$ is odd, we see that these agree. Hence, ${S}_{\sigma}^1$ is a symmetry of $\gup 1$.
\hfill $\triangle$
\end{ex}

\begin{ex}
We again look at the complex given by Figure \ref{fig_label}, with the vertex-labelling on the left. Suppose $\fdown 1:C_0(X) \rightarrow C_0(X)$ is componentwise with each component given by the odd function $f$. We then have 
\begin{align}
 \gdown 1 (x) = B_1^{\intercal} \fdown 1 B_1(x)  
&=  \left( \begin{array}{l}
f(x_1 + x_3 + x_4) + f(x_1 - x_2) \\
f(x_2 + x_3 - x_5)  + f(x_2 - x_1)  \\
f(x_3 + x_1 + x_4)  + f(x_3 + x_2 - x_5) \\
f(x_4 + x_1 + x_3) + f(x_4 + x_5) \\
f(x_5 - x_2 - x_3 ) + f(x_5 + x_4) \\
\end{array} \right) \, .
\end{align}
Let $\tilde{S}_{\sigma}^1$ and ${S}_{\sigma}^1$ be as in Example \ref{firstexamwithsdsig}. One verifies that $ \gdown 1$ does not commute with $\tilde{S}_{\sigma}^1$, whereas it does commute with ${S}_{\sigma}^1$.
\hfill $\triangle$
\end{ex}

It follows from the examples above that $\tilde{S}_{\sigma}^d$ is not necessarily a symmetry of  $ \gup d$  and  $ \gdown d$. Instead, the symmetry of the dynamics is represented by ${S}_{\sigma}^d$, as the following theorem tells us.

\begin{thr}\label{mainthronsym}
Let $\sigma$ be a symmetry of the complex $X$ and suppose $ \fup d $  and $ \fdown d $ are componentwise with each component given by the same odd function. Then we have 
\begin{equation}
 \gup d \circ S^d_{\sigma}  = S^d_{\sigma}  \circ   \gup d  \text{ and } \gdown d   \circ S^d_{\sigma}  = S^d_{\sigma}  \circ   \gdown d  \, .
\end{equation}
\end{thr}

The key lemma for proving Theorem \ref{mainthronsym} is given by:

\begin{lem}\label{commutsss}
For all $d>0$ and $\sigma \in \mathcal{S}_X$ it holds that $S^{d-1}_{\sigma}B_d = B_d S^{d}_{\sigma}$.
\end{lem}

\begin{proof}
As before, we identify a $d$-simplex $s = [i_0, \dots, i_d]$ with its corresponding basis element in $C_d(X)$. On the one hand, we see that
\begin{align}
S^{d-1}_{\sigma}B_d [i_0, \dots, i_d] &= S^{d-1}_{\sigma}\sum_{k=0}^d(-1)^k [i_0, \dots, \widehat{i_k}, \dots, i_d] \\ \nonumber
&= \sum_{k=0}^d(-1)^k \sgn(A_{s} \setminus \{i_k\}, \sigma) \sigma([i_0, \dots, \widehat{i_k}, \dots, i_d])\, .
\end{align}
On the other hand, we have 
\begin{align}
B_d S^{d}_{\sigma}[i_0, \dots, i_d] &= B_d  \sgn(A_{s}, \sigma) \sigma([i_0, \dots i_d]) \\ \nonumber
&=  \sgn(A_{s}, \sigma) \sum_{k=0}^d (-1)^{s(\sigma(i_k))} \sigma([i_0, \dots, \widehat{i_k}, \dots i_d])\, ,
\end{align}
where $s(\sigma(i_k))$ denotes the position of $\sigma(i_k)$ when we order the elements in $\{\sigma(i_0), \dots, \sigma(i_d) \}$. Hence, we are done if we can show that
\begin{equation}
(-1)^{s(\sigma(i_k))}(-1)^k \sgn(A_{s} \setminus \{i_k\}, \sigma) =  \sgn(A_{s}, \sigma)\, .
\end{equation}
However, this is precisely the statement we have showed in the proof of Lemma \ref{lemforlabel}. That is, we can order 
$$(\sigma(i_0), \dots, \sigma(i_d))$$
by first bringing $\sigma(i_k)$ to its right place $s(\sigma(i_k))$, which accounts for the term $(-1)^{|s(\sigma(i_k))- k|} = (-1)^{s(\sigma(i_k))}(-1)^k$. After this we order the remaining terms, which accounts for $\sgn(A_{s} \setminus \{i_k\}, \sigma)$. This completes the proof.\end{proof}

We will also use the following result.

\begin{lem}\label{transposeisinverse}
For each $\sigma \in \mathcal{S}_X$ and $d \geq 0$ we have $(S^d_{\sigma})^{\intercal} = S^d_{\sigma^{-1}}$.
\end{lem}

\begin{proof}
As before, we interpret the elements of $X_d$ as a basis for $C_d(X)$. The standard inner-product on $C_d(X)$ then satisfies $s^{\intercal}t = \delta_{s,t}$ for all $s,t \in X_d$. It follows that 
\begin{align}\label{innerproductofS}
s^{\intercal}S^d_{\sigma}t = s^{\intercal} \sgn(A_{t}, \sigma) \sigma(t) = \sgn(A_{t}, \sigma)\delta_{s, \sigma(t)}\, .
\end{align}
We therefore also have
\begin{align}\label{innerproductofStrans}
t^{\intercal}S^d_{\sigma^{-1}}s =  \sgn(A_{s}, \sigma^{-1})\delta_{t, \sigma^{-1}(s)} = \sgn(A_{s}, \sigma^{-1})\delta_{s, \sigma(t)} = \sgn(A_{\sigma(t)}, \sigma^{-1})\delta_{s, \sigma(t)} \, .
\end{align}
Now, from Lemma \ref{lemmaonsgn} it follows that  $\sgn(A_{\sigma(t)}, \sigma^{-1})\sgn(A_{t}, \sigma) =  \sgn(A_{t}, \id) = 1$. As we have $\sgn(A_{\sigma(t)}, \sigma^{-1}), \sgn(A_{t}, \sigma) \in \{1,-1\}$, we see that in fact $\sgn(A_{\sigma(t)}, \sigma^{-1}) = \sgn(A_{t}, \sigma)$. This shows that equations \eqref{innerproductofS} and \eqref{innerproductofStrans} agree, and we conclude that indeed $(S^d_{\sigma})^{\intercal} = S^d_{\sigma^{-1}}$.
\end{proof}

\begin{proof}[Proof of Theorem \ref{mainthronsym}]
Note that $ \fup d$ commutes with each map $S^{d+1}_{\sigma}$, as the latter can be written as the product of a permutation matrix and a diagonal matrix with diagonal entries in $\{1, -1\}$ (in any order). Likewise,  $\fdown d $ commutes with each map $S^{d-1}_{\sigma}$. By Lemma \ref{commutsss} we have $S^{d-1}_{\sigma}B_d = B_d S^{d}_{\sigma}$ for all $\sigma \in \mathcal{S}_X$, from which it also follows that
\begin{align}
B_d^{\intercal}S^{d-1}_{\sigma}  =      B_d^{\intercal}(S^{d-1}_{\sigma^{-1}})^{\intercal}      =      (S^{d-1}_{\sigma^{-1}}B_d)^{\intercal} = (B_d S^{d}_{\sigma^{-1}})^{\intercal} = (S^{d}_{\sigma^{-1}})^{\intercal}B_d^{\intercal} = S^{d}_{\sigma}B_d^{\intercal} \, ,
\end{align}
using Lemma \ref{transposeisinverse}.
The statement now follows from an easy calculation as in the proof of Theorem \ref{mainthronlabel}.
\end{proof}

\section{Invariant spaces}\label{sec:invariant}

Next, we investigate synchrony and anti-synchrony spaces for $\gdown d = B_{d}^\intercal\fdown d B_{d}$ and  $\gup d = B_{d+1}\fup d B_{d+1}^\intercal$. That is, we look at invariant spaces given by identities of the form $x_s = x_t$,  $x_s = -x_t$ and $x_s = 0$ for d-simplices $s$ and $t$. Motivated by the results of Section \ref{sec:realization}, we will only consider the case where $\fdown d$ or $\fup d$ is component-wise, and ask about such spaces that are invariant for any choice of $\fdown d$ or $\fup d$ in a particular structural family. To describe this family, we fix a partition $\mathcal{P}= (P_1, \dots, P_k)$ of the $D$-simplices, and write 
\begin{equation}\label{descripCompD}
\cw {\mathcal{P}}{D} = \left\{     \begin{aligned}
  & \mathcal{F}: C_{D}(X) \rightarrow C_{D}(X)  \text{ component-wise with odd components, }   \\
  & \mathcal{F}_{s} =  \mathcal{F}_{t} \text{ if } s,t \in P_i \text{ for some } i \in \{1, \dots, k\}.
  \end{aligned}   \right\} \, .
\end{equation}
The interpretation of this family is that some of the higher- or lower-dimensional simplices through which communication occurs (where $D=d+1$ or $D=d-1$) are of identical type, while others are different in nature. We are then interested in all (anti)-synchrony spaces $W \subseteq C_d(X)$ satisfying $B_{d}^\intercal\fdown d B_{d}(W) \subseteq W$ for all $\fdown d \in \cw {\mathcal{P}}{d-1}$, for a given partition $\mathcal{P}$ of the $(d-1)$-simplices. Similarly, we ask for all such spaces satisfying $B_{d+1}\fup d B_{d+1}^\intercal(W) \subseteq W$ for all $\fup d \in \cw {\mathcal{Q}}{d+1}$, for a given partition $\mathcal{Q}$ of the $(d+1)$-simplices. \\
Note that we have restricted to odd components in the description \eqref{descripCompD} of $\cw {\mathcal{P}} D$. This is motivated by the results of Section \ref{sec:orientation}, where we found that this assumption guarantees that different orientations of the simplex give conjugate systems. As these conjugacies are by diagonal matrices with diagonal entries in $\{1, -1\}$, we argue that there is no real distinction between synchrony and anti-synchrony spaces. This is the reason we consider both, and we will simply refer to them as anti-synchrony spaces from here on out. The technical details behind this argument are given by the following straightforward result.

\begin{lem}
Let $F: \R^n \rightarrow  \R^n$ be a vector field and $G = A\circ F \circ A^{-1}: \R^n \rightarrow  \R^n$ the push-forward of $F$ by the invertible linear map $A: \R^n \rightarrow  \R^n$. In other words, $A$ conjugates $F$ and $G$. The assignment $W \mapsto AW$ induces a bijection between the linear spaces $W \subseteq \R^n$ that are invariant under $F$ and those that are invariant under $G$.
\end{lem}

\begin{proof}
It is clear that the assignment $W \mapsto AW$ is a bijection of the set of subspaces of $\R^n$, with inverse given by $W \mapsto A^{-1}W$. As we may also write $F = A^{-1}\circ G \circ A$, it suffices to show that $AW$ is invariant under $G$ if  $W$ is invariant under $F$. However, this follows immediately as $G(AW) = AF(A^{-1}AW) = AF(W) \subseteq AW$. This completes the proof.
\end{proof}

\noindent Note that for $W$ a synchrony space (so defined by identities of the form $x_s = x_t$ for d-simplices $s$ and $t$) and $A$ a diagonal map with diagonal entries in $\{1, -1\}$, the space $AW$ might only be an anti-synchrony space. We start with some useful definitions and notation. 

\begin{defi}
Let $\mathcal{K}$ be a finite set of colors, and denote by $\Z\mathcal{K}$ the free $\Z$-module with basis $\mathcal{K}$. That is, elements of $\Z\mathcal{K}$ are given by formal sums of the form
\[\sum_{k \in \mathcal{K}} a_k k \, \text{ with }  \, a_k \in \Z \, \text{ for all } \,  k \in \mathcal{K}.\]
Addition in $\Z\mathcal{K}$ is component-wise, i.e. 
\begin{equation}
\sum_{k \in \mathcal{K}} a_k k + \sum_{k \in \mathcal{K}} b_k k = \sum_{k \in \mathcal{K}} (a_k+b_k) k \, ,
\end{equation}
and we may multiply by an element $r \in \Z$ by setting 
\begin{equation}
r\left(\sum_{k \in \mathcal{K}} a_k k\right) = \sum_{k \in \mathcal{K}} ra_k k \, .
\end{equation}
We also denote by $\mathbf{B}_{\mathcal{K}} \subseteq \Z\mathcal{K}$ the finite subset consisting of only the elements $k := 1k$ and $-k := -1k$ for all $k \in \mathcal{K}$, together with $0 := \sum_{k \in \mathcal{K}} 0 k$. \\
An \emph{anti-coloring} of the $d$-simplices is an assignment $K_d: X_d \rightarrow \mathbf{B}_{\mathcal{K}}$.\\

\noindent Given a $d$-simplex $t \in X_{d}$ and a $(d-1)$-simplex $e \in X_{d-1}$, we define the \emph{sign induced on $e$ by $t$} as
\begin{align}
I_{e}^t = \begin{cases}
0 &\text{ if } e \text{ is not in the boundary } t; \\ 
1 &\text{ if } e \text{ appears in } \partial_{d} t \text { with a positive sign; } \\
-1 &\text{ if } e \text{ appears in } \partial_{d} t \text { with a negative sign. } 
\end{cases}
\end{align}
Note that another way of defining $I_{e}^t$ is by stating that
\begin{equation}\label{maincharacterIet}
\partial_{d} t = \sum_{e \in X_{d-1}} I_{e}^t e\, ,
\end{equation}
which makes it straightforward to find $I_{e}^t$ by seeing what vertex is left out of $t$ to obtain $e$ (if $I_{e}^t \not= 0$). In particular, we simply have $I_e^t = (B_d)_{et}$, though we use a separate notation to emphasize the simple construction of the scalar $I_e^t$ (which is independent of the rest of the simplex) and to avoid cluttered notation. \\

\noindent Given an anti-coloring  $K_d$ on $X_d$, we define the \emph{induced anti-colorings} $\ccup{d}$ and $\ccdown{d}$ on $X_{d+1}$ and $X_{d-1}$, respectively. The coloring $\ccup{d}: X_{d+1} \rightarrow \Z\mathcal{K}$ is given by
\begin{equation}
\ccup{d}(t) = \sum_{e \in X_{d}} I_{e}^t K_d(e)\, ,
\end{equation}
for all $t \in X_{d+1}$. In other words, $\ccup{d}$ counts the colors of the boundary simplices in the signed way induced by $\partial_{d+1}$. Similarly, we define $\ccdown{d}: X_{d-1} \rightarrow \Z\mathcal{K}$  by
\begin{equation}
\ccdown{d}(v) = \sum_{e \in X_{d}} I_{v}^e K_d(e)\, .
\end{equation}
Note that $\ccup{d}$ and $\ccdown{d}$ in general do not take values in $\mathbf{B}_{\mathcal{K}}$. 
\hfill $\triangle$
\end{defi}

\begin{remk}
If we write $\bar{K}_d = (\dots, K_d(e), \dots) \in \Z\mathcal{K}^{X_d}$ for the corresponding vector of colors, with similar definitions for $\ccupv{d}$ and $\ccdownv{d}$, then we get the compact descriptions
\begin{equation}
{\ccupv{d}} = B_{d+1}^\intercal \bar{K}_d \text{ and }  \ccdownv{d} = B_{d} \bar{K}_d\, .
\end{equation}
These are very useful for calculating induced anti-colorings directly, and provide a heuristic interpretation of these colorings as (co)-boundaries of the given coloring. \hfill $\triangle$
\end{remk}

The main technical definition we will need is the following. The terminology is motivated by that of a similar construction by Stewart et al~\cite{stewart2003symmetry}.

\begin{defi}\label{upanddownbalanced}
Let $\mathcal{P} = (P_1, \dots, P_k)$ be a partition of $X_{d+1}$. An anti-coloring $K_d$ of $X_d$ is called \emph{up-balanced with respect to $\mathcal{P}$} if the following holds:
\begin{enumerate}
\item If $e,f \in X_d$ satisfy $K_d(e) = K_d(f)$ then for all classes $P_i$ of $\mathcal{P}$ and all elements $a \in \Z\mathcal{K}$ we have
\begin{equation}\label{checkupsyn}
\sum\limits_{\substack{t \in P_i \\ \ccup{d}(t) = a}} I_{e}^t - \sum\limits_{\substack{t \in P_i \\ \ccup{d}(t) = -a}} I_{e}^t = \sum\limits_{\substack{t \in P_i \\ \ccup{d}(t) = a}} I_{f}^t - \sum\limits_{\substack{t \in P_i \\ \ccup{d}(t) = -a}} I_{f}^t \, .
\end{equation}
\item If $e,f \in X_d$ satisfy $K_d(e) = -K_d(f)$ then for all classes $P_i$ of $\mathcal{P}$ and all elements $a \in \Z\mathcal{K}$ we have
\begin{equation}
\sum\limits_{\substack{t \in P_i \\ \ccup{d}(t) = a}} I_{e}^t - \sum\limits_{\substack{t \in P_i \\ \ccup{d}(t) = -a}} I_{e}^t = - \sum\limits_{\substack{t \in P_i \\ \ccup{d}(t) = a}} I_{f}^t + \sum\limits_{\substack{t \in P_i \\ \ccup{d}(t) = -a}} I_{f}^t \, .
\end{equation}
\end{enumerate}
Similarly, let $\mathcal{Q} = (Q_1, \dots, Q_l)$ be a partition of $X_{d-1}$. An anti-coloring $K_d$ of $X_d$ is called \emph{down-balanced with respect to $\mathcal{Q}$} if the following holds:
\begin{enumerate}
\item If $e,f \in X_d$ satisfy $K_d(e) = K_d(f)$ then for all classes $Q_i$ of $\mathcal{Q}$ and all elements $a \in \Z\mathcal{K}$ we have
\begin{equation}
\sum\limits_{\substack{v \in Q_i \\ \ccdown{d}(v) = a}} I_{v}^e - \sum\limits_{\substack{v \in Q_i \\ \ccdown{d}(v) = -a}} I_{v}^e = \sum\limits_{\substack{v \in Q_i \\ \ccdown{d}(v) = a}} I_{v}^f - \sum\limits_{\substack{v \in Q_i \\ \ccdown{d}(v) = -a}} I_{v}^f\, .
\end{equation}
\item If $e,f \in X_d$ satisfy $K_d(e) = -K_d(f)$ then for all classes  $Q_i$ of $\mathcal{Q}$ and all elements $a \in \Z\mathcal{K}$ we have
\begin{equation}
\sum\limits_{\substack{v \in Q_i \\ \ccdown{d}(v) = a}} I_{v}^e - \sum\limits_{\substack{v \in Q_i \\ \ccdown{d}(v) = -a}} I_{v}^e = - \sum\limits_{\substack{v \in Q_i \\ \ccdown{d}(v) = a}} I_{v}^f + \sum\limits_{\substack{v \in Q_i \\ \ccdown{d}(v) = -a}} I_{v}^f\, .
\end{equation} \hfill $\triangle$
\end{enumerate}
\end{defi}

\noindent Note that if $K_d(e) = 0$ we have $K_d(e) = K_d(e)$ and $K_d(e) = -K_d(e)$, so that both points have to be satisfied for $e = f$. In case of up-balanced with respect to $\mathcal{P}$ we then precisely require
\begin{equation}
\sum\limits_{\substack{t \in P_i \\ \ccup{d}(t) = a}} I_{e}^t - \sum\limits_{\substack{t \in P_i \\ \ccup{d}(t) = -a}} I_{e}^t = 0 \, \,  \text{ and so } \quad \sum\limits_{\substack{t \in P_i \\ \ccup{d}(t) = a}} I_{e}^t = \sum\limits_{\substack{t \in P_i \\ \ccup{d}(t) = -a}} I_{e}^t \, ,
\end{equation}
for all classes $P_i$ and all $a \in \Z\mathcal{K}$, with a similar equality for down-balanced colorings. \\

\begin{remk}
The expression 
\begin{equation}
\sum\limits_{\substack{t \in P_i \\ \ccup{d}(t) = a}} I_{e}^t - \sum\limits_{\substack{t \in P_i \\ \ccup{d}(t) = -a}} I_{e}^t 
\end{equation}
from Definition \ref{upanddownbalanced} can be obtained from $B_{d+1}$ and $\ccupv{d}$ as follows. Denote by $[\ccupv{d}]_{a, P_i}$ the color-vector obtained from $\ccupv{d}$ by setting to zero all entries belonging to simplices $t$ that are not in $P_i$ or that do not satisfy $\ccup{d}(t) \in \{a, -a\}$. In other words, we set
\begin{equation}
 ([\ccupv{d}]_{a, P_i})_t = \delta_{\ccup{d}(t) \in \{a, -a\}} \delta_{t \in P_i}\ccup{d}(t) = \delta_{\ccup{d}(t) = a} \delta_{t \in P_i}a -  \delta_{\ccup{d}(t) = -a} \delta_{t \in P_i} a\, .
 \end{equation}
for all $t \in X_{d+1}$. We then simply have 
 \begin{align}
 (B_{d+1}([\ccupv{d}]_{a, P_i}))_e &= \sum_{t \in X_{d+1}} (B_{d+1})_{et} ([\ccupv{d}]_{a, P_i})_t \\ \nonumber
& =\sum_{t \in X_{d+1}} I_e^t (\delta_{\ccup{d}(t) = a} \delta_{t \in P_i}a -  \delta_{\ccup{d}(t) = -a} \delta_{t \in P_i} a) = \bigg(\sum\limits_{\substack{t \in P_i \\ \ccup{d}(t) = a}} I_{e}^t - \sum\limits_{\substack{t \in P_i \\ \ccup{d}(t) = -a}} I_{e}^t \bigg)a\, .
 \end{align}
for all $e \in X_d$. The condition of being up-balanced is therefore equivalent to certain equalities among the entries of the color-vectors $B_{d+1}([\ccupv{d}]_{a, P_i})$, for all $a \in \Z\mathcal{K}$ and all classes $P_i$. In the same way, we find
 \begin{align}
 (B_{d}^{\intercal}([\ccdownv{d}]_{a, Q_i}))_e = \bigg(\sum\limits_{\substack{v \in Q_i \\ \ccdown{d}(v) = a}} I_{v}^e - \sum\limits_{\substack{v \in Q_i \\ \ccdown{d}(v) = -a}} I_{v}^e \bigg)a\, ,
 \end{align}
for all $e \in X_d$, where $[\ccdownv{d}]_{a, Q_i}$ is obtained from $\ccdownv{d}$ by setting all entries to zero, except for those corresponding to $v \in Q_i$ with $\ccdown{d}(v) \in \{-a,a\}$. \hfill $\triangle$
\end{remk}

\noindent To an anti-coloring $K_d$ on $X_d$ we may associate an anti-synchrony space $\Delta_{K_d} \subseteq C^d(X)$ defined by the equalities $x_s = x_t$ if $K_d(s) = K_d(t)$, $x_s = -x_t$ if $K_d(s) = -K_d(t)$ and so $x_s  = 0$ if $K_d(s) = 0$. Conversely, any anti-synchrony space has a corresponding anti-coloring in this way. 

\begin{remk}
In exactly the same way as for anti-synchrony spaces, an anti-coloring $K_d$ on $X_d$ induces a subspace $\Delta^{\mathcal{K}}_{K_d}$ of $\Z\mathcal{K}^{X_d}$. Again, we say that $x \in \Delta^{\mathcal{K}}_{K_d}$ if $x_s = x_t$ whenever $K_d(s) = K_d(t)$ and $x_s = -x_t$ if $K_d(s) = -K_d(t)$. This gives a very brief description of up- and down-balanced colorings: $K_d$ is up-balanced with respect to $\mathcal{P} = (P_1, \dots, P_k)$ if  $B_{d+1}([\ccupv{d}]_{a, P_i}) \in \Delta^{\mathcal{K}}_{K_d}$ for all classes $P_i$ and all $a \in \Z\mathcal{K}$. An algorithm for determining whether $K_d$ is up-balanced with respect to $\mathcal{P}$ is therefore given by:
\begin{enumerate}
\item Construct the induced coloring ${\ccupv{d}} = B_{d+1}^\intercal \bar{K}_d$;
\item For all classes $P_i$ and all $a \in \Z\mathcal{K}$ (with $a$ appearing in ${\ccupv{d}}$), construct $[\ccupv{d}]_{a, P_i}$ by setting all entries to zero corresponding to $t \in X_{d+1}$ that are not in $P_i$ or that satisfy $\ccup{d}(t) \notin \{a, -a\}$;
\item Verify if $B_{d+1}([\ccupv{d}]_{a, P_i}) \in \Delta^{\mathcal{K}}_{K_d}$ for all $[\ccupv{d}]_{a, P_i}$. If so, $K_d$ is up-balanced.
\end{enumerate}
Similarly, $K_d$ is down-balanced with respect to $\mathcal{Q} = (Q_1, \dots, Q_l)$ if  $B_{d}^{\intercal}([\ccdownv{d}]_{a, Q_i}) \in \Delta^{\mathcal{K}}_{K_d}$ for all classes $Q_i$ and all $a \in \Z\mathcal{K}$. We get the analogous algorithm for down-balanced colorings:
\begin{enumerate}
\item Construct the induced coloring ${\ccdownv{d}} = B_{d} \bar{K}_d$;
\item For all classes $Q_i$ and all $a \in \Z\mathcal{K}$ (with $a$ appearing in ${\ccdownv{d}}$), construct $[\ccdownv{d}]_{a, Q_i}$ by setting all entries to zero corresponding to $v \in X_{d-1}$ that are not in $Q_i$ or that satisfy $\ccdown{d}(v) \notin \{a, -a\}$;
\item Verify if $B_{d}^{\intercal}([\ccdownv{d}]_{a, Q_i}) \in \Delta^{\mathcal{K}}_{K_d}$ for all $[\ccdownv{d}]_{a, Q_i}$. If so, $K_d$ is down-balanced.  \hfill $\triangle$
\end{enumerate}
\end{remk}

The following result classifies the anti-colorings for which the corresponding anti-synchrony space is invariant.

\begin{thr}\label{mainonantisynchrony}
Let $K_d$ be an anti-coloring on $X_d$ and denote by $\Delta_{K_d} \subseteq C_d(X)$ the corresponding anti-synchrony subspace. We fix partitions $\mathcal{P}$ and $\mathcal{Q}$ of $X_{d+1}$ and $X_{d-1}$, respectively. Then
\begin{itemize}
\item The space $\Delta_{K_d} $ is invariant under $\gup d = B_{d+1}\fup d B_{d+1}^\intercal$ for all $\fup d  \in \cw {\mathcal{P}}{d+1}$, if and only if $K_d$ is up-balanced with respect to $\mathcal{P}$.
\item The space $\Delta_{K_d} $ is invariant under $\gdown d := B_{d}^\intercal\fdown d B_{d}$ for all $\fdown d  \in \cw {\mathcal{Q}}{d-1}$, if and only if $K_d$ is down-balanced with respect to $\mathcal{Q}$.
\end{itemize}
\end{thr}

Before we prove Theorem \ref{mainonantisynchrony}, we first illustrate the result with some examples.

\begin{ex}\label{exampledowncoloredwithdiamond}
We return to the first system considered in Example \ref{examlabeleffect2}, corresponding to the left side of Figure  \ref{fig_label}  and given by 
\begin{align}\label{leftODEtriangles2invar}
B_2^\intercal \fdown2B_2(x) 
&= \left( \begin{array}{l}
2f(x_1) +  f(x_1 - x_2) \\
2f(x_2) + f(x_2 - x_1) 
\end{array}  \right) \, .
\end{align}
Here $x_1$ describes the state of the $2$-simplex $[1,2,3]$ and $x_2$ that of  $[1,3, 4]$. They communicate via the edges, which are given the trivial partition $\mathcal{Q} = (X_1)$. Only six colorings correspond to different possible anti-synchrony spaces. These are 
\begin{align}
&K_{2,1}([1,2,3]) = c, K_{2,1}([1,3,4]) = d; \quad
&K_{2,2}([1,2,3]) = c, K_{2,2}([1,3,4]) = c; \\ \nonumber
&K_{2,3}([1,2,3]) = c, K_{2,3}([1,3,4]) = -c; \quad
&K_{2,4}([1,2,3]) = 0, K_{2,4}([1,3,4]) = 0; \\ \nonumber
&K_{2,5}([1,2,3]) = c, K_{2,5}([1,3,4]) = 0; \quad
&K_{2,6}([1,2,3]) = 0, K_{2,6}([1,3,4]) = c, 
\end{align}
for distinct colors $c,d$. Note that $K_{2,1}$ is trivially down-balanced, with corresponding anti-synchrony space the full phase space. For $K_{2,2}$ and $K_{2,3}$ we calculate the corresponding induced colorings by
\begin{equation}
{\ccdownv{2,2}} = \begin{pmatrix}
1 & 0 \\
1 & 0 \\
-1 & 1 \\
0 & -1 \\
0 & 1 \\
\end{pmatrix}
\begin{pmatrix}
c \\  c
\end{pmatrix} = 
\begin{pmatrix}
c \\
c \\
0\\
-c \\
c \\
\end{pmatrix} 
\text{ and }
{\ccdownv{2,3}} = \begin{pmatrix}
1 & 0 \\
1 & 0 \\
-1 & 1 \\
0 & -1 \\
0 & 1 \\
\end{pmatrix}
\begin{pmatrix}
c \\  -c
\end{pmatrix} = 
\begin{pmatrix}
c \\
c \\
-2c\\
c \\
-c \\
\end{pmatrix} \, .
\end{equation}
Verifying if $K_{2,2}$ is down-balanced is easiest, as ${\ccdownv{2,2}}$ only has entries in $\{-c,0,c\}$. We thus have $[{\ccdownv{2,2}}]_{c, X_1} = {\ccdownv{2,2}}$, and a calculation shows that
\begin{equation}
B_2^{\intercal}[{\ccdownv{2,2}}]_{c, X_1} = \begin{pmatrix}
1 & 1 & -1 & 0 & 0 \\ 
0 & 0 & 1 & -1 & 1  \\
\end{pmatrix}
\begin{pmatrix}
c \\
c \\
0\\
-c \\
c \\
\end{pmatrix} =
\begin{pmatrix}
2c \\
2c \\
\end{pmatrix} \in \Delta_{K_{2,2}}^{\mathcal{K}}\, .
\end{equation}
This shows that $K_{2,2}$ is down-balanced, and one verifies that $\Delta_{K_{2,2}} = \{x_1 = x_2\}$ is indeed invariant for any system of the form \eqref{leftODEtriangles2invar}. For  $K_{2,3}$ we only have to consider 
\begin{equation}
[{\ccdownv{2,3}}]_{c, X_1} = \begin{pmatrix}
c \\
c \\
0\\
c \\
-c \\
\end{pmatrix} \, \text{ and } \,
[{\ccdownv{2,3}}]_{2c, X_1} = \begin{pmatrix}
0 \\
0 \\
-2c\\
0 \\
0 \\
\end{pmatrix} \, .
\end{equation}
We find
\begin{equation}
B_2^{\intercal}[{\ccdownv{2,3}}]_{c, X_1} = \begin{pmatrix}
1 & 1 & -1 & 0 & 0 \\ 
0 & 0 & 1 & -1 & 1  \\
\end{pmatrix}
\begin{pmatrix}
c \\
c \\
0\\
c \\
-c \\
\end{pmatrix} =
\begin{pmatrix}
2c \\
-2c \\
\end{pmatrix} \in \Delta_{K_{2,3}}^{\mathcal{K}}\, ,
\end{equation}
and 
\begin{equation}
B_2^{\intercal}[{\ccdownv{2,3}}]_{2c, X_1} = \begin{pmatrix}
1 & 1 & -1 & 0 & 0 \\ 
0 & 0 & 1 & -1 & 1  \\
\end{pmatrix}
\begin{pmatrix}
0 \\
0 \\
-2c\\
0 \\
0 \\
\end{pmatrix} =
\begin{pmatrix}
2c \\
-2c \\
\end{pmatrix} \in \Delta_{K_{2,3}}^{\mathcal{K}}\, .
\end{equation}
Hence, $K_{2,3}$ is also down-balanced. Correspondingly, the space $\{x_1 = -x_2\}$ is invariant under any system of the form \eqref{leftODEtriangles2invar} (provided $f$ is odd).  The coloring $K_{2,4}$ corresponds to the zero-space, which is clearly invariant. As for $K_{2,5}$, we find
\begin{equation}
{\ccdownv{2,5}} = \begin{pmatrix}
1 & 0 \\
1 & 0 \\
-1 & 1 \\
0 & -1 \\
0 & 1 \\
\end{pmatrix}
\begin{pmatrix}
c \\  0
\end{pmatrix} = 
\begin{pmatrix}
c \\
c \\
-c\\
0 \\
0 \\
\end{pmatrix} = [{\ccdownv{2,5}}]_{c, X_1}\, .
\end{equation}
However, 
\begin{equation}
B_2^{\intercal}[{\ccdownv{2,5}}]_{c, X_1} = \begin{pmatrix}
1 & 1 & -1 & 0 & 0 \\ 
0 & 0 & 1 & -1 & 1  \\
\end{pmatrix}
\begin{pmatrix}
c \\
c \\
-c\\
0 \\
0 \\
\end{pmatrix} =
\begin{pmatrix}
3c \\
-c \\
\end{pmatrix} \notin \Delta_{K_{2,5}}^{\mathcal{K}}\, .
\end{equation}
This shows that $K_{2,5}$ is not down-balanced. A similar argument shows that $K_{2,6}$ is not down-balanced, so that the only anti-synchrony spaces invariant for any system of the form \eqref{leftODEtriangles2invar} are $\R^2, \{x_1 = x_2\},  \{x_1 = -x_2\}$ and $0$. \hfill $\triangle$
\end{ex}

\begin{ex}
We return to the setting of Example \ref{exampledowncoloredwithdiamond}, but we now consider the non-trivial partition of $X_1$ given by $\mathcal{Q} = (Q_1, Q_2)$ with $Q_1 = \{e_1\}$ and $Q_2 = \{e_2, e_3, e_4, e_5\}$. A general system of the form $B_2^\intercal \fdown2B_2$ with $ \fdown2 \in \cw {\mathcal{Q}}{1}$ is given by
\begin{align}\label{leftODEtriangles2invar-2qq}
B_2^\intercal \fdown2B_2(x) 
&= \left( \begin{array}{l}
f(x_1) + g(x_1) + f(x_1 - x_2) \\
2f(x_2) + f(x_2 - x_1) 
\end{array}  \right) \, ,
\end{align}
for odd functions $f,g: \R \rightarrow \R$. The situation of Example \ref{exampledowncoloredwithdiamond} can be seen as a special case of what we consider here, by setting $f = g$. Hence, we only have to consider the down-balanced colorings from this previous example. The spaces $\R^2$ and $\{0\}$ are trivially invariant for any system of the form \eqref{leftODEtriangles2invar-2qq}, so that $K_{2,1}$ and $K_{2,4}$ remain down-balanced. Recall that
\begin{equation}
{\ccdownv{2,2}} = 
\begin{pmatrix}
c \\
c \\
0\\
-c \\
c \\
\end{pmatrix} 
\text{ and }
{\ccdownv{2,3}} =  
\begin{pmatrix}
c \\
c \\
-2c\\
c \\
-c \\
\end{pmatrix}, \text{ so that }
[{\ccdownv{2,2}}]_{c, Q_1} =  [{\ccdownv{2,3}}]_{c, Q_1} =  
\begin{pmatrix}
c \\
0 \\
0\\
0 \\
0 \\
\end{pmatrix}\, .
\end{equation}
We have 
\begin{equation}
B_2^{\intercal}[{\ccdownv{2,2}}]_{c, Q_1} = B_2^{\intercal}[{\ccdownv{2,3}}]_{c, Q_1} = \begin{pmatrix}
1 & 1 & -1 & 0 & 0 \\ 
0 & 0 & 1 & -1 & 1  \\
\end{pmatrix}
\begin{pmatrix}
c \\
0 \\
0\\
0 \\
0 \\
\end{pmatrix} =
\begin{pmatrix}
c \\
0 \\
\end{pmatrix} \, ,
\end{equation}
which is neither an element of $\Delta_{K_{2,2}}^{\mathcal{K}} = \{ x_1 = x_2\}$ nor of $\Delta_{K_{2,3}}^{\mathcal{K}} = \{ x_1 = -x_2\}$. Hence, $K_{2,2}$ and $K_{2,3}$ are not down-balanced. \hfill $\triangle$
\end{ex}

\begin{ex}
To further illustrate Theorem \ref{mainonantisynchrony}, let us consider the classical case where $n$ nodes $\{1, \dots, n\}$ are communicating via a set of edges $E \subseteq \{[i,j] \mid i,j \in \{1, \dots, n\}\, , i<j\}$. We furthermore assume $\mathcal{P} = \{E\}$ is the trivial partition, so that we are interested in anti-synchrony spaces that are invariant for all ODEs of the form
\begin{equation}\label{simpleexampleequatiie}
\dot{x}_i = \sum_{[i,j]/[j,i] \in E} f(x_i - x_j) \text { for all } i \in \{1, \dots, n\}\, ,
\end{equation}
with $x_i \in \R$ and for odd functions $f: \R \rightarrow \R$. We will restrict our attention to synchrony spaces, so those given by 
\begin{equation}
\Delta_{K_0} = \{x \in \R^n \mid x_i = x_j \text{ if } K_0(i) = K_0(j)\}
\end{equation}
for some map $K_0$ from the nodes to a set of colors $\mathcal{K}$. \\
The induced coloring  $\ccup{0}$ is given by  $\ccup{0}([i,j]) = K_0(i) - K_0(j)$, so that in particular $\ccup{0}([i,j]) = 0$ if and only if $K_0(i) = K_0(j)$. We will consider Equation \eqref{checkupsyn} for each color combination $a$, where we note that for $a = 0$ there is nothing to check. \\
Let us therefore fix a node $i$ with color $k =K_0(i)$, another color $\tilde{k}$, and set $a = k-\tilde{k}$. The only way in which we can have $I^e_i \not= 0$ and $\ccup{0}(e) = a$ for an edge $e$ is when $e = [i,j]$ where $i<j$ and $K_0(j) = \tilde{k}$. This is because  $I^e_i \not= 0$ forces $i$ to appear in $e$, and $\ccup{0}(e) = a$ means $i$ appears first in $e$. We then have $I^e_i = 1$, and as a result we find 
\begin{equation}\label{checkupsynex1}
\sum\limits_{\substack{e \in E \\ \ccup{0}(e) = a}} I_{i}^e = \#\{[i,j] \in E\, , K_0(j) = \tilde{k}\} \, .
\end{equation}
Similarly, we only have $I^e_i \not= 0$ and $\ccup{0}(e) = -a$ for an edge $e$ when $e = [j,i]$ with $K_0(j) = \tilde{k}$, in which case $I_{i}^e = -1$. Hence, 
\begin{equation}\label{checkupsynex2}
\sum\limits_{\substack{e \in E \\ \ccup{0}(e) = -a}} I_{i}^e = -\#\{[j,i] \in E\, , K_0(j) = \tilde{k}\} \, .
\end{equation}
We therefore find that
\begin{equation}\label{checkupsynex3}
\sum\limits_{\substack{e \in E \\ \ccup{0}(e) = a}} I_{i}^e - \sum\limits_{\substack{e \in E \\ \ccup{0}(e) = -a}} I_{i}^e 
\end{equation}
counts precisely the number of edges between $i$ and a node of color $\tilde{k}$.  Theorem \ref{mainonantisynchrony} thus tells us that $K_0$ defines an invariant synchrony space if and only if for any two nodes of the same color $k$, and for any color $\tilde{k} \not= k$, the number of connections to a node of color $\tilde{k}$ is the same. \\
This result is well-known and can also be seen from Equation \eqref{simpleexampleequatiie} directly. To this end, write $x = (x_1, \dots, x_n) = (x_{K_0(1)}, \dots, x_{K_0(n)})$ for a general point in $\Delta_{K_0}$. Restricted to this space, the equations of motion for $x_i$ become 
\begin{equation}\label{checkupsynex3q}
\dot{x}_i = \sum_{[i,j]/[j,i] \in E} f(x_{K_0(i)} - x_{K_0(j)})\, ,
\end{equation}
which counts the number of connections between $i$ and a node of a different color $\tilde{k}$ as the number of terms  $f(x_{K_0(i)} - x_{\tilde{k}})$. Hence, we indeed get identical equations of motion for $x_i$ and $x_j$ with $K_0(i) = K_0(j)$, precisely when the number of connections to nodes of other colors agrees. \\
The case of general anti-synchrony spaces is slightly more involved, as elements in $\Z\mathcal{K}$ of the form $a = \pm 2k$ have to be taken into account as well. This corresponds to the case where we have an edge between nodes of opposing colors, which accounts for a term $f(\pm2x_k)$ in Equation \eqref{checkupsynex3q}. \hfill $\triangle$
\end{ex}

In the proof of Theorem \ref{mainonantisynchrony} below, given an anti-coloring $K_d$ we denote a point on the corresponding anti-synchrony space $\Delta_{K_d}$ by $x = (x_1, \dots, x_n) = (x_{K_d(1)}, \dots, x_{K_d(n)})$, with the convention that $x_{-k} = -x_k$ for $k \in \mathbf{B}_{\mathcal{K}}$ (and so $x_0 = 0$). Generalizing this, given $x \in \Delta_{K_d}$ we may write $x_a := \sum_{k \in \mathcal{K}} a_k x_k$ for any element $a =  \sum_{k \in \mathcal{K}} a_k k \in \Z\mathcal{K}$ with $a_k \in \Z$. More precisely, this is under the assumption that $K_d$ reaches all colors in $\mathcal{K}$, which may be assumed as additional colors have no meaning. It then trivially holds that $x_{ra} = rx_{a}$ for all $r \in \Z$ and $x_a+x_b = x_{a+b}$ for all $a, b \in  \Z\mathcal{K}$. Note that this notation is only well-defined when we have $x \in \Delta_{K_d}$. \\

\begin{proof}[Proof of Theorem \ref{mainonantisynchrony}]
We start with invariance under $\gup d = B_{d+1}\fup d B_{d+1}^\intercal$. Given a partition $\mathcal{P} = (P_1, \dots, P_k)$ of $X_{d+1}$, we write $(\fup d(x))_t = g_{i(t)}(x_t)$ if we have $t \in P_{i(t)} \subseteq X_{d+1}$. For $x \in \Delta_{K_d}$ and $e \in X_d$, we see that
\begin{align}\label{bigcalcinvar1}
(B_{d+1}\fup d B_{d+1}^\intercal(x))_e &= \sum_{t \in X_{d+1}} (B_{d+1})_{et}(\fup d B_{d+1}^\intercal(x))_t =  \sum_{t \in X_{d+1}} I_{e}^{t}g_{i(t)}([B_{d+1}^\intercal(x)]_t) \\ \nonumber
 &=  \sum_{t \in X_{d+1}} I_{e}^{t}g_{i(t)}\left(\sum_{e' \in X_{d}}[B_{d+1}^\intercal]_{te'}x_{e'}\right)  =  \sum_{t \in X_{d+1}} I_{e}^{t}g_{i(t)}\left(\sum_{e' \in X_{d}}[B_{d+1}]_{e't}x_{K_d(e')}\right)  \\ \nonumber
 &=  \sum_{t \in X_{d+1}} I_{e}^{t}g_{i(t)}\left(\sum_{e' \in X_{d}}I_{e'}^t x_{K_d(e')}\right) =  \sum_{t \in X_{d+1}} I_{e}^{t}g_{i(t)}\left(x_{\ccup{d}(t)}\right) \, ,
\end{align}
where in the last step we have used our notation $x_a$ for $a \in \Z\mathcal{K}$. We may further rewrite Equation \eqref{bigcalcinvar1} as 
\begin{align}\label{bigcalcinvar3}
 \sum_{t \in X_{d+1}} I_{e}^{t}g_{i(t)}\left(x_{\ccup{d}(t)}\right) =  \sum_{i = 1}^k \sum_{t \in P_i} I_{e}^{t}g_{i}\left(x_{\ccup{d}(t)}\right)  = \sum_{i = 1}^k \sum_{a \in \Z\mathcal{K}} \sum\limits_{\substack{t \in P_i \\ \ccup{d}(t) = a}}  I_{e}^{t}g_{i}\left(x_{a}\right) \, .
\end{align}
Now, since all $g_i$ are odd and therefore satisfy $g_i(0) = 0$,  we may remove $0 \in \Z\mathcal{K}$ from the middle sum in the last expression of \eqref{bigcalcinvar3}. We moreover write $[\Z\mathcal{K}]$ for a full set of representatives of $\Z\mathcal{K}\setminus \{0\}$ under the equivalence relation $a \sim b \Longleftrightarrow a = \pm b$, $a, b \in \Z\mathcal{K}\setminus \{0\}$. This allows us to write 
\begin{align}\label{bigcalcinvar4}
\sum_{i = 1}^k \sum_{a \in \Z\mathcal{K}} \sum\limits_{\substack{t \in P_i \\ \ccup{d}(t) = a}}  I_{e}^{t}g_{i}\left(x_{a}\right) &= \sum_{i = 1}^k \sum_{a \in [\Z\mathcal{K}]} \left( \sum\limits_{\substack{t \in P_i \\ \ccup{d}(t) = a}}  I_{e}^{t}g_{i}\left(x_{a}\right) + \sum\limits_{\substack{t \in P_i \\ \ccup{d}(t) = -a}}  I_{e}^{t}g_{i}\left(x_{-a}\right) \right) \\ \nonumber
&= \sum_{i = 1}^k \sum_{a \in [\Z\mathcal{K}]} \left( \sum\limits_{\substack{t \in P_i \\ \ccup{d}(t) = a}}  I_{e}^{t}g_{i}\left(x_{a}\right) - \sum\limits_{\substack{t \in P_i \\ \ccup{d}(t) = -a}}  I_{e}^{t}g_{i}\left(x_{a}\right) \right) \\ \nonumber
&= \sum_{i = 1}^k \sum_{a \in [\Z\mathcal{K}]} \left( \sum\limits_{\substack{t \in P_i \\ \ccup{d}(t) = a}}  I_{e}^{t} - \sum\limits_{\substack{t \in P_i \\ \ccup{d}(t) = -a}}  I_{e}^{t} \right) g_{i}\left(x_{a}\right) \, .
\end{align}
In conclusion, we find
\begin{align}\label{bigcalcinvar5}
(B_{d+1}\fup d B_{d+1}^\intercal(x))_e = \sum_{i = 1}^k \sum_{a \in [\Z\mathcal{K}]} \left( \sum\limits_{\substack{t \in P_i \\ \ccup{d}(t) = a}}  I_{e}^{t} - \sum\limits_{\substack{t \in P_i \\ \ccup{d}(t) = -a}}  I_{e}^{t} \right) g_{i}\left(x_{a}\right) \, .
\end{align}
Now, the odd functions $g_i$ can all be chosen freely, and since no two elements in $[\Z\mathcal{K}]$ differ by a minus sign, we may freely choose the values of $g_{i}\left(x_{a}\right)$ as well. More precisely, every $x_a$ is a linear combination of the $x_{K_d(e')}$, and so for general $x \in \Delta_{K_d}$ we have $x_a \not= x_b \not= 0$ for all $a \not= b$ appearing in the finite set $\mathcal{X}_d^{\uparrow} = \{\pm \ccup{d}(t)\mid t \in X_{d+1}\} \setminus \{0\}$.  We may construct a map $g_i$ with prescribed values on each $x_a$ for $a \in \mathcal{X}_d^{\uparrow} \cap [\Z\mathcal{K}]$, as well as minus these values on the $-x_a$. We may then make $g_i$ odd by redefining it as $x \mapsto 1/2(g_i(x) - g_i(-x))$. \\
With this freedom, we conclude that for $e,f \in X_d$ we have
\begin{equation}
(B_{d+1}\fup d B_{d+1}^\intercal(x))_e = \pm(B_{d+1}\fup d B_{d+1}^\intercal(x))_f
\end{equation}
for all $x \in \Delta_{K_d}$ and $\fup d  \in \cw {\mathcal{P}}{d+1}$, if and only if 
 \begin{equation}
 \sum\limits_{\substack{t \in P_i \\ \ccup{d}(t) = a}}  I_{e}^{t} - \sum\limits_{\substack{t \in P_i \\ \ccup{d}(t) = -a}}  I_{e}^{t} =  \pm \sum\limits_{\substack{t \in P_i \\ \ccup{d}(t) = a}}  I_{f}^{t}  \mp \sum\limits_{\substack{t \in P_i \\ \ccup{d}(t) = -a}}  I_{f}^{t} \, ,
\end{equation}
for all $i \in \{1, \dots, k\}$ and $a \in [\Z\mathcal{K}]$, and so for all $a \in \Z\mathcal{K}$. This proves the first part of the theorem.

The second part is nearly identical. Let $\mathcal{Q} = (Q_1, \dots, Q_l)$ be a partition of $X_{d-1}$ and write $(\fdown d(x))_v = h_{j(v)}(x_v)$ whenever we have $v \in Q_{j(v)} \subseteq X_{d-1}$ and $x \in C_{d-1}(X)$.  A calculation as before shows that
\begin{align}\label{bigcalcinvar1paart2}
(B_{d}^{\intercal}\fdown d B_{d}(x))_e =  \sum_{v \in X_{d-1}} I_{v}^{e}h_{j(v)}\left(x_{\ccdown{d}(v)}\right) \, ,
\end{align}
for all $x \in \Delta_{K_d}$. This may further be rewritten as 
\begin{align}\label{bigcalcinvar2paart2}
\sum_{v \in X_{d-1}} I_{v}^{e}h_{j(v)}\left(x_{\ccdown{d}(v)}\right) &= \sum_{j= 1}^l \sum_{a \in \Z\mathcal{K}} \sum\limits_{\substack{v \in Q_j \\ \ccdown{d}(v) = a}}  I_{v}^{e}h_{j}\left(x_{a}\right) \\ \nonumber &= \sum_{j= 1}^l \sum_{a \in [\Z\mathcal{K}]} \left( \sum\limits_{\substack{v \in Q_j \\ \ccdown{d}(v) = a}}  I_{v}^{e} - \sum\limits_{\substack{v \in Q_j \\ \ccdown{d}(v) = -a}}  I_{v}^{e} \right) h_{j}\left(x_{a}\right) \, .
\end{align}
We again conclude that
\begin{equation}
(B_{d}^{\intercal}\fdown d B_{d}(x))_e = \pm(B_{d}^{\intercal}\fdown d B_{d}(x))_f
\end{equation}
for all $x \in \Delta_{K_d}$ and $\fdown d  \in \cw {\mathcal{Q}}{d-1}$, if and only if 
 \begin{equation}
 \sum\limits_{\substack{v \in Q_j \\ \ccdown{d}(v) = a}}  I_{v}^{e} - \sum\limits_{\substack{v \in Q_j \\ \ccdown{d}(v) = -a}}  I_{v}^{e} =  \pm \sum\limits_{\substack{v \in Q_j \\ \ccdown{d}(v) = a}}  I_{v}^{f}  \mp \sum\limits_{\substack{v \in Q_j \\ \ccdown{d}(v) = -a}}  I_{v}^{f} 
\end{equation}
for all $j \in \{1, \dots, l\}$ and $a \in \Z\mathcal{K}$. This completes the proof.
\end{proof}

\appendix

\section{Some proofs}\label{app:proofs}

\begin{proof}[Proof of Lemma~\ref{lem:decomp}]
The second and third decomposition of $C_d(X)$ are a result of the rank-nullity theorem. More precisely, we have
$$\dim(\im(B_d^\intercal)) + \dim(\ker(B_d)) = \dim(\im(B_d)) + \dim(\ker(B_d)) = \dim(C_d(X)) \, ,$$
so that 
\begin{align}\label{dimperp1}
\dim(\im(B_d^\intercal)) = \dim(C_d(X)) - \dim(\ker(B_d)) = \dim(\ker(B_d)^\perp)\, .
\end{align}
Now, if we have $x \in \ker(B_d)$ and $y \in \im(B_d^\intercal)$, then we may write $y = B_d^\intercal z$ for some $z \in C_{d-1}(X)$. It follows that
$$\langle x, y\rangle = \langle x, B_d^\intercal z\rangle = \langle B_dx, z\rangle = \langle 0, z\rangle = 0\, .$$
Hence, we have  $\im(B_d^\intercal) \subseteq \ker(B_d)^\perp$ and so by Equation \eqref{dimperp1}, $\im(B_d^\intercal) = \ker(B_d)^\perp$. This shows the orthogonal decomposition
\begin{equation}\label{decompinproof1}C_d(X) =  \im(B_d^\intercal) \oplus \ker(B_d)\, .\end{equation}
Applying the same argument with $B_d$ replaced by $B_{d+1}^\intercal$ gives
\begin{equation}\label{decompinproof2}C_d(X) =  \im(B_{d+1}) \oplus \ker(B_{d+1}^\intercal)\, .\end{equation}
We next focus on the decomposition of $\ker(B_d)$. From $B_dB_{d+1} = 0$ we see that $\im(B_{d+1}) \subseteq \ker(B_d)$. We also have $W_d \subseteq \ker(B_d)$, by definition of the former space. Given $x \in \ker(B_d)$, it follows from Equation \eqref{decompinproof2} that we may write $x = y + z$ for some $y \in  \im(B_{d+1})$ and $z \in \ker(B_{d+1}^\intercal)$. From $\im(B_{d+1}) \subseteq \ker(B_d)$ we conclude that $z = x-y \in \ker(B_d)$, and so $z \in W_d$. This shows that 
$$\ker(B_d) =  \im(B_{d+1})+ W_d\, .$$
From $W_d \subseteq \ker(B_{d+1}^\intercal)$ and Equation \eqref{decompinproof2} we finally conclude that we have the orthogonal decomposition 
\begin{equation}\label{decompinproof3}\ker(B_d)  =  \im(B_{d+1}) \oplus W_d\, .\end{equation}
A similar argument, but with $B_d$ replaced by $B_{d+1}^\intercal$, shows that 
\begin{equation}\label{decompinproof4} \ker(B_{d+1}^\intercal)  = \im(B_d^\intercal) \oplus W_d \, . \end{equation}
Finally, the decomposition 
\begin{equation}\label{decompinproof5} C_d(X) =  \im(B_{d+1}) \oplus \im(B_d^\intercal) \oplus W_d \, , \end{equation}
follows from equations \eqref{decompinproof1} and \eqref{decompinproof3} (and likewise from \eqref{decompinproof2} and \eqref{decompinproof4}).
\end{proof}

\begin{proof}[Proof of Lemma \ref{lem:oplus}]
To show existence of $L^+$, we write
\begin{equation}
S = \im(L)^{\perp} \subseteq W \text{ and } T = \ker(L)^{\perp} \subseteq V\, .
\end{equation}
It follows that any element $w \in W$ may be uniquely written as $w = s + Lt$, with $s \in S$ and $t \in T$. Here we use the fact that $L$ is injective when restricted to $T$. We then set $L^+(s+Lt) = t$, which is readily seen to satisfy the three conditions of the lemma. \\

\noindent To show uniqueness, we first note that trivially $\ker(L^+) \subseteq \ker(LL^+)$. Conversely, if  $LL^+w = 0$ then $L^+w = L^+LL^+w = 0$, so that $w \in \ker(L^+)$. It follows that $\ker(L^+) = \ker(LL^+)$. Similarly we have $\im(LL^+) \subseteq \im(L)$, and for any $w = Lv \in \im(L)$ we have $w = LL^+Lv \in \im(LL^+)$. We conclude that $\im(LL^+) = \im(L)$, and by analogous arguments we find $\ker(L) = \ker(L^+L)$ and $\im(L^+L) = \im(L^+)$. Since $LL^+$ and $L^+L$ are both symmetric, we find the orthogonal decompositions
  \begin{equation*}
    V = \ker (L) \oplus \im (L^+),\quad\quad W = \ker (L^+) \oplus \im (L).
  \end{equation*}
It follows that any map $L^+$ satisfying the three conditions of Definition \ref{moorepenrose} has to vanish on $(\im (L))^{\perp}$.  Given $w = Lv \in \im(L)$, we get  $L^+(w) = L^+L(v)$ and so $L(L^+(w)) = LL^+L(v) = Lv$. As the equation $Lx = y$ has at most one  solution $x \in  \im (L^+) =  (\ker (L))^{\perp}$ for a given $y \in W$, this fully determines $L^+$, which shows uniqueness. \\

\noindent Finally, it is easy to see that $(L^+)^*$ satisfies all the conditions for $(L^*)^{+}$, so that the two indeed agree by uniqueness.

\end{proof}

\section{Exempla bonissima}

\subsection{Guckenheimer--Holmes example}

We want to demonstrate an application of the result in Section~\ref{sec:realization}.  Let us assume that $X$ is the ``all-to-all'' simplicial complex on 4 vertices (basically the ``full tetrahedron''), and let us consider the ``triangle flow''.  Thus the flow we consider is
\begin{equation}\label{eq:triangle2}
  \theta_2 ' = B_2^\intercal \fdown2(B_2 \theta_2),
\end{equation}
where $\fdown2\colon \R^6\to\R^6$ can be whatever we like.  Note that we have 
\begin{equation*}
  B_2 = \begin{pmatrix} 1&1&0&0\\-1&0&1&0\\0&-1&-1&0\\1&0&0&1\\0&1&0&-1\\0&0&1&1\end{pmatrix}.
\end{equation*}
We see from inspection that $B_2$ has rank three and thus there are  three degrees of freedom for this vector field.  
Let us now consider the three-dimensional flow given by the Guckenheimer--Holmes cycle~\cite{guckenheimer1988structurally}, which in coordinates can be written
\begin{align*}
  \dot x &= x\left(\mu - (ax^2 + by^2 + cz^2)\right),\\
  \dot y &= y\left(\mu - (ay^2 + bz^2 + cx^2)\right),\\
  \dot z &= z\left(\mu - (az^2 + bx^2 + cy^2)\right),
\end{align*}
and let us define $\widetilde H\colon \R^3\to\R^3$ as the function given by the right-hand side.  Following the process in Section~\ref{sec:realization}, let us adjoin a single constant flow to this vector field by adding the equation $\dot w = 0$, and thus induces a vector field on $\R^4$.

We now define $M$ in terms of its inverse:
\begin{equation*}
  M^{-1} := \begin{pmatrix}
    1&1&0&1\\1&0&0&-1\\0&0&1&1\\0&1&1&-1
  \end{pmatrix} \implies M = \frac14 \begin{pmatrix}
    1&3&1&-1\\2&-2&-2&2\\-1&1&3&1\\1&-1&1&-1
  \end{pmatrix}.
\end{equation*}
We constructed this by finding three vectors that span $\im B_2^\intercal$ (they are the first, third, and last rows of $B_2$ and clearly independent) and adjoining a vector in the kernel of $B_2$.  From this it follows that $M$ is an isomorphism from $\im B^\intercal \oplus \ker B$ to $\R^3\oplus \R$.

From this, we can define $G(x) = M^{-1}H(Mx)$, and achieve this $G$ by $\fdown 2(x) = A^\intercal G(Ax)$, where $A$ is the pseudoinverse of $B_2$.

We plot an example trajectory of this system in Figure~\ref{fig:GH} with parameters $\mu = 1$, $a = 1$, $b=0.55$, $c=1.5$.  We chose the initial conditions as follows:  we chose a random point in a neighborhood of $(1/2,1/2,1/2)$ and adjoined the value $1/3$ to it, then multiplied by $M^{-1}$ and plotted the trajectory on the left.  We also plot $M$ times the trajectory on the right, and see an expected example of the Guckenheimer--Holmes cycle with a constant component of $1/3$ attached to it.
\begin{figure}[th]
\begin{center}
\includegraphics[width=0.7\textwidth]{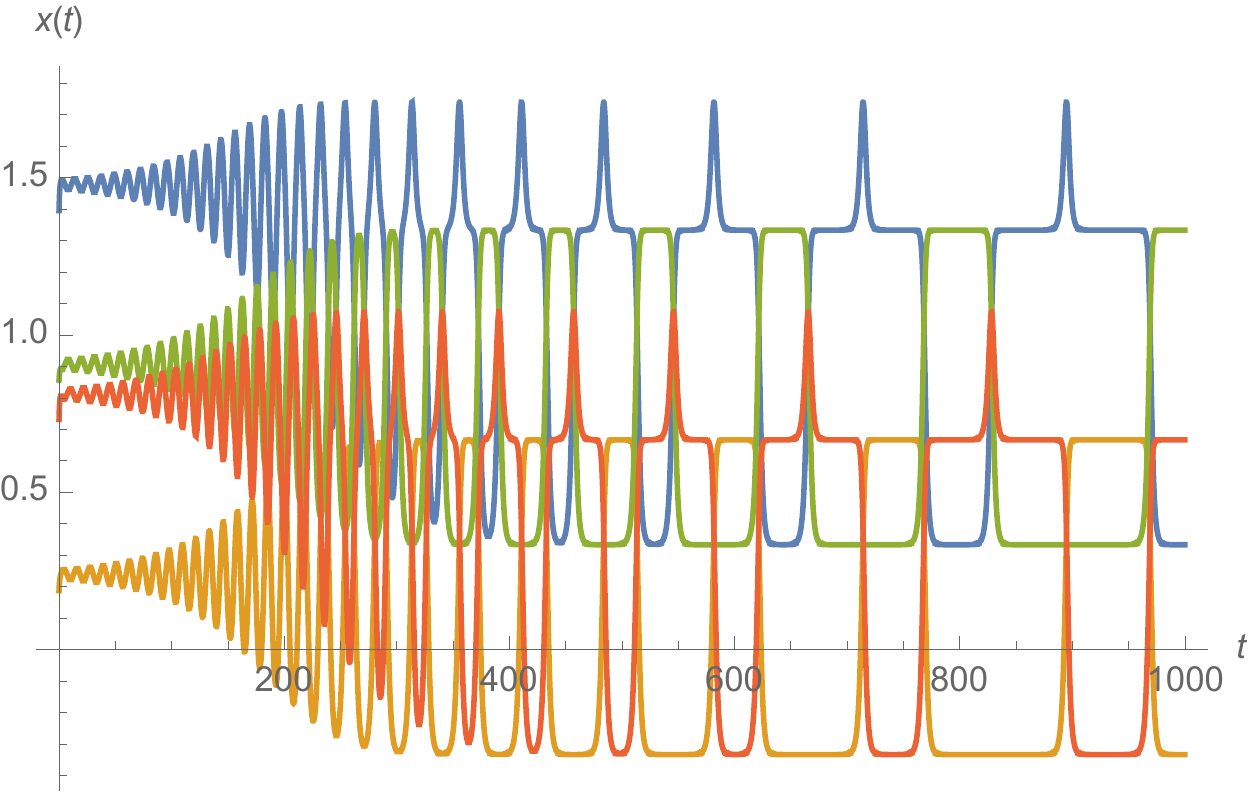}
\includegraphics[width=0.7\textwidth]{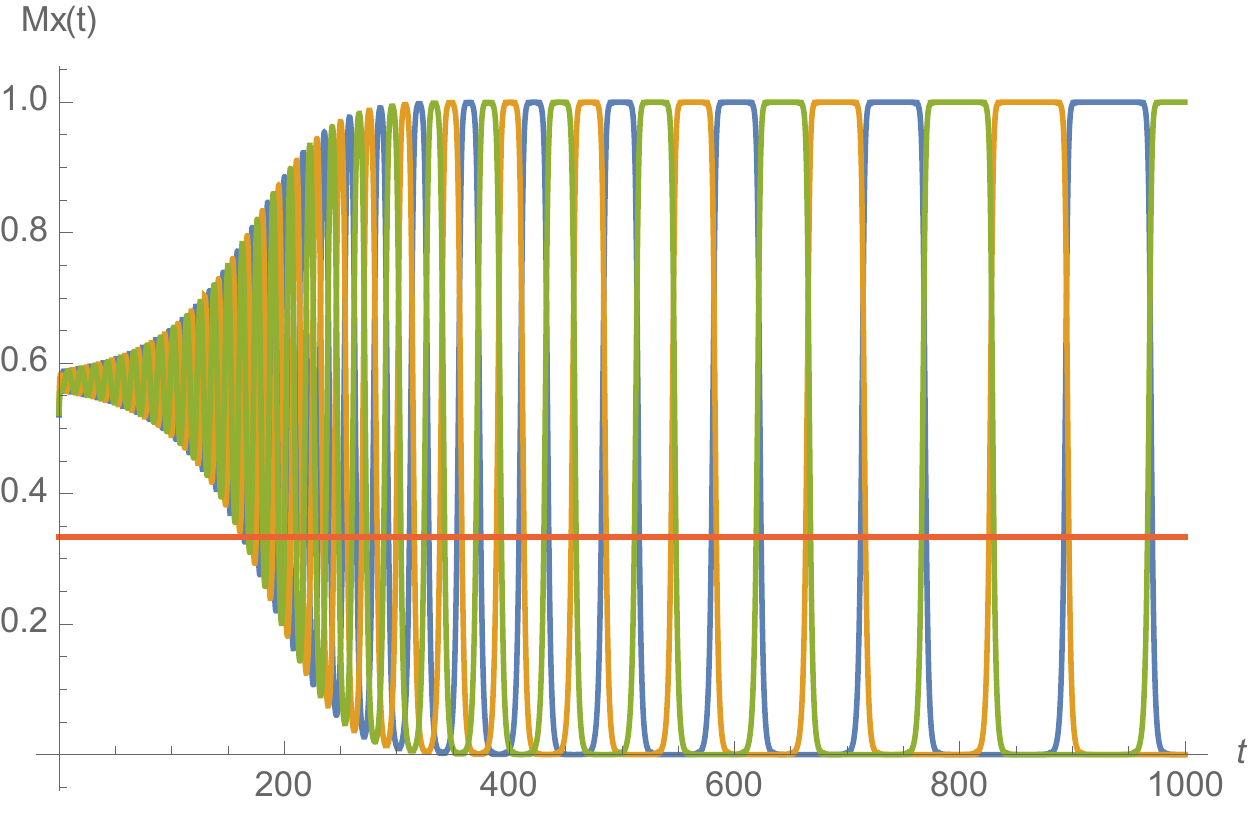}
\caption{An evolution of the system $G(x)$.  On the top we plot the exact simulation, and on the bottom we plot the solution left-multiplied by the matrix $M$.}
\label{fig:GH}
\end{center}
\end{figure}

\subsection{Lorenz--Sel'kov system}

Here we want to give an example of a chaotic system that lives on a simplicial complex.  In particular we will choose a combination of two systems:  the first will be the well-known Lorenz system~\cite{lorenz1963deterministic}, and the second will be the Sel'kov model for glycolysis~\cite{sel1968self} (mostly famous in the math community for being a realistic system that has a Hopf bifurcation, but in this example we choose parameters to get an attracting limit cycle).

\begin{figure}[th]
\begin{center}
\includegraphics[width=0.6\textwidth]{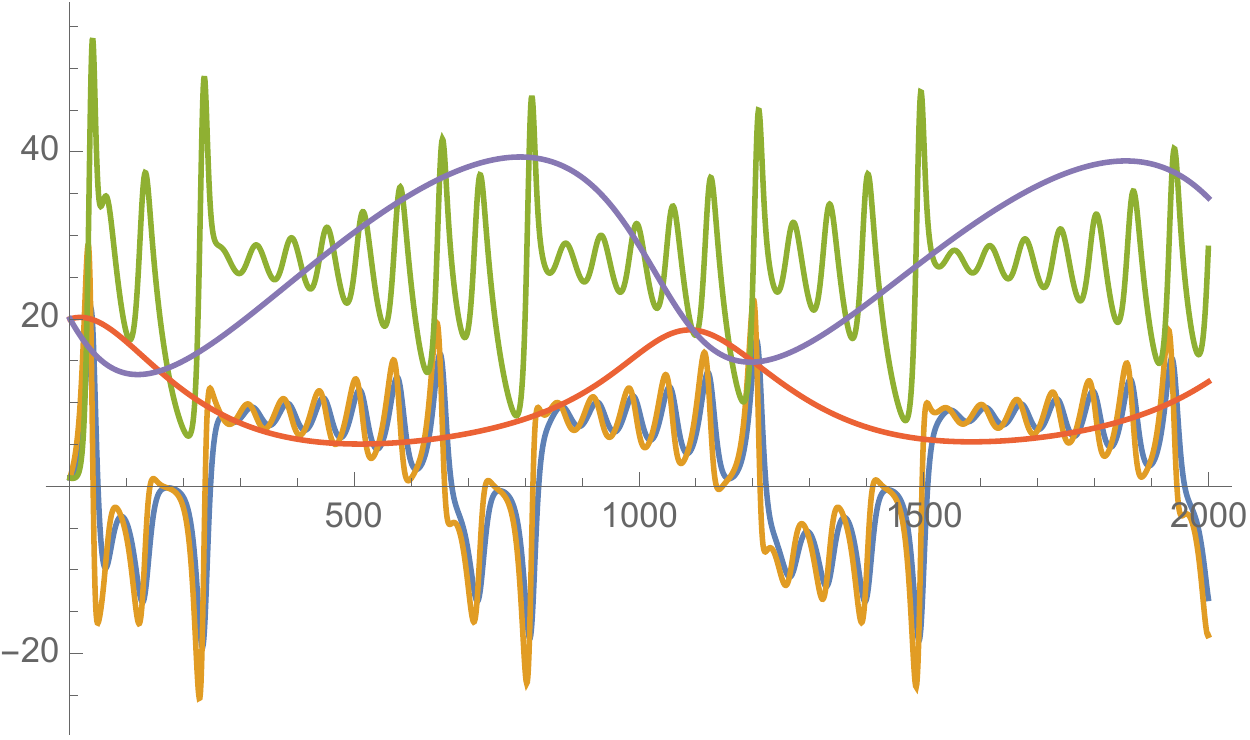}
\includegraphics[width=0.6\textwidth]{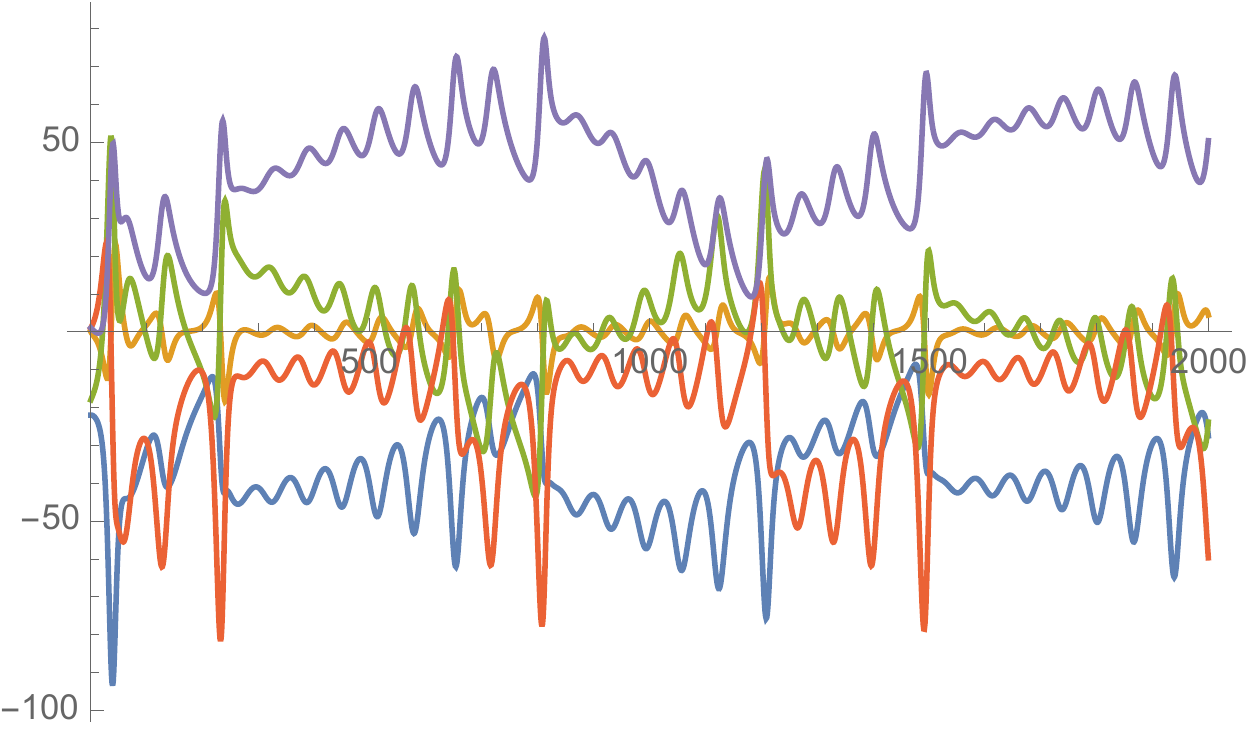}
\end{center}
\caption{One example of the system~\eqref{eq:LS} with $\sigma = 10$, $\beta = 8/3$, $\rho=28$, $a = 0.1$, and $b=1.5$  .}
\label{fig:LS}
\end{figure}

Let us consider the simplicial complex denoted in the left of Figure~\ref{fig_label}, where we always choose orientations to point ``upward'' in the vertex number.  In particular, this gives
\begin{equation*}
  \partial t_1 = e_1 + e_2 - e_3,\quad \partial t_2 = e_3 - e_4 + e_5
\end{equation*}
and
\begin{align*}
  \partial e_1 &= v_2 - v_1,\quad \partial e_2 = v_3 - v_2,\quad \partial e_3 = v_3 - v_1,\\
  \partial e_4 &= v_4 - v_1,\quad \partial e_5 = v_4 - v_3.
\end{align*}

We can compute that
\begin{equation*}
  B_1 = \begin{pmatrix}
    -1&0&-1&-1&0\\1&-1&0&0&0\\0&1&1&0&-1\\0&0&0&1&1
  \end{pmatrix},\quad
  B_2 = \begin{pmatrix}
    1&0\\1&0\\-1&1\\0&-1\\0&1
  \end{pmatrix}
\end{equation*}

It is not hard to see that $B_1$ has rank 3 and $B_2$ has rank 2, so according to the terminology of Definition~\ref{def:abn}, this means that $r_1 = 3$, $r_2 = 2$, and $n_1 = 5$, so that any system on $C_1(X)$ will be a system of type $(3,2,5)$; namely a system on $\R^3$ and an independent system on $\R^2$.

Now let us recall the Lorenz system:
\begin{equation*}
  \frac{d}{dt}\begin{pmatrix}
    x\\y\\z
  \end{pmatrix}= 
  \begin{pmatrix}
    \sigma(y-x)\\x(\rho-z)-y\\xy-\beta x
  \end{pmatrix}
\end{equation*}
and the Selkov model:
\begin{equation*}
  \frac{d}{dt}\begin{pmatrix}
    \xi\\\eta 
  \end{pmatrix}=\begin{pmatrix}
    -\xi + a \eta + \xi^2 \eta/400\\b-a\eta -\xi^2\eta/400
  \end{pmatrix}.
\end{equation*}
The standard presentation of this model is the same as above, but without dividing by the factors of 400 in the nonlinearity.  We did this division so that the amplitude of the limit cycle in the Sel'kov system would be comparable to the amplitude of the Lorenz oscillator with the standard parameters.  We can embed these into $\R^5$ by taking the direct sum of the dynamics:
\begin{equation}\label{eq:LS}
  \frac{d}{dt}\begin{pmatrix}
    x\\y\\z\\\xi\\\eta
  \end{pmatrix}
  = H(x,y,z,\xi,\eta) := \begin{pmatrix}
    \sigma(y-x)\\x(\rho-z)-y\\xy-\beta x\\-\xi + a \eta + \xi^2 \eta/400\\b-a\eta -\xi^2\eta/400
  \end{pmatrix}.
\end{equation}

We now want to construct the isomorphism that sends this system into a flow of the form~\eqref{eq:1}.  We can basically do the following:  the first three rows of $B_1$ are linearly independent, and thus form a basis for the column space of $B_1^\intercal$, and the columns of $B_2$ are linearly independent.  So we want $M^{-1}$ to be the $5\times 5$ matrix that sends the standard basis to these five vectors, and $M$ to be its inverse, so
\begin{equation*}
  M^{-1} = \left(
\begin{array}{ccccc}
 -1 & 0 & -1 & -1 & 0 \\
 1 & -1 & 0 & 0 & 0 \\
 0 & 1 & 1 & 0 & -1 \\
 1 & 1 & -1 & 0 & 0 \\
 0 & 0 & 1 & -1 & 1 \\
\end{array}
\right),\quad M = \frac18 \left(
\begin{array}{ccccc}
 -1 & 4 & 1 & 3 & 1 \\
 -1 & -4 & 1 & 3 & 1 \\
 -2 & 0 & 2 & -2 & 2 \\
 -5 & -4 & -3 & -1 & -3 \\
 -3 & -4 & -5 & 1 & 3 \\
\end{array}
\right)
\end{equation*}

From this, we can define $G(x) = M^{-1}H(Mx)$, and this is a flow on the edge space of this simplex.

We plot the results of a simulation of this system in Figure~\ref{fig:LS}.  We plot the dynamics of $H$ and the dynamics of $G$.  In the $H$ system, the Lorenz oscillator and the limit cycle are completely disjoint, and we can see this in the timeseries.  In the $G$ system (the system on the simplicial complex), these two get mixed up and we get something like a ``breathing'' Lorenz oscillator.
%

\end{document}